\title{A theta expression of the Hilbert modular functions for  $\sqrt{5}$ via the  periods of $K3$ surfaces}
\author{Atsuhira Nagano}
\documentclass{article}
\usepackage{mathrsfs,bm,graphics,amsfonts,amsthm,amssymb,amsmath,amscd,booktabs}
\usepackage[pdflatex]{graphicx}
\def\bigzerou{\smash{\lower1.7ex\hbox{\b 0}}}

\newtheorem{thm}{Theorem}[section]
\newtheorem{df}{Definition}[section]
\newtheorem{lem}{Lemma}[section]
\newtheorem{prop}{Proposition}[section]
\newtheorem{rem}{Remark}[section]

\newtheorem{cor}{Corollary}[section]

\def\comment#1{{ }}
\makeatletter
  
      \@addtoreset{equation}{section}
        
      \makeatletter

\begin{document}
\maketitle

\begin{abstract}
In this paper, 
we give an extension of the classical story of the elliptic modular function to the Hilbert modular case for $\mathbb{Q}(\sqrt{5})$.
We construct the period mapping for a family $\mathcal{F}=\{S(X,Y)\}$ of $K3 $ surfaces with $2$ complex parameters $X$ and $Y$.
The inverse correspondence  of the period mapping gives a system of generators  of Hilbert modular functions for $\mathbb{Q}(\sqrt{5})$.
 Moreover, we show an explicit  expression of this inverse correspondence by theta constants.
\end{abstract}

\footnote[0]{Running head: Hilbert modular  functions via  $K3$ surfaces}
\setlength{\baselineskip}{14 pt}

\section*{Introduction}

The symmetric Hilbert modular surface $(\mathbb{H}\times \mathbb{H})/\langle PSL(2,\mathcal{O}_K),\tau\rangle$,
where $\mathcal{O}_K$ is the ring of integers in a real quadratic field $K$
and $\tau$ exchanges the factors of $\mathbb{H}\times\mathbb{H}$,
 gives the moduli space for the  family  $\mathcal{F}_K=\{A\}$ of the principally polarized Abelian surfaces with an extra endomorphism structure $K(\subset {\rm End}^0 (A))$.

In the classical theory, the elliptic modular function $\lambda (z)$ on the moduli space $\mathbb{H}/\Gamma(2)$ is given by the inverse of the multivalued period mapping for a family of elliptic curves. 
This period mapping gives the Schwarz mapping of 
the Gauss hypergeometric differential equation $\displaystyle E\Big( \frac{1}{2},\frac{1}{2},1\Big)$.
It is important that the modular function $\lambda (z)$ has an explicit  expression given by the Jacobi theta constants. 

For the Hilbert modular cases, 
although there are various studies on the structure of the field of modular functions and the ring of modular forms
 (for example Gundlach \cite{Gundlach}, Hirzebruch \cite{Hirzebruch} and M\"uller \cite{Muller}),
still now,
to the best of the author's knowledge,
there has not appeared an explicit expression of  Hilbert modular functions as an inverse correspondence of the period mapping for a family of algebraic varieties.
In this paper, 
we give an extension of the above classical story to the Hilbert modular functions for $K=\mathbb{Q}(\sqrt{5})$ by using a family of $K3$ surfaces
that gives the same variation of  Hodge structures of weight $2$ with the family $\mathcal{F}_K$ of Abelian surfaces.
Namely,
we show that
the inverse  of the period mapping for our family of $K3$ surfaces
gives   Hilbert modular functions for $\mathbb{Q}(\sqrt{5})$.  
Moreover, 
we obtain an explicit theta expression of this inverse correspondence.
As our method, we use the fact that our  period integrals of $K3$ surfaces satisfy a system of partial differential equations determined in \cite{Nagano}.

Our results is obtained as a combined work  with \cite{Nagano}
and  based on  the results  of Hirzebruch \cite{Hirzebruch} and M\"uller \cite{Muller}  also.

In this paper, we consider the family 
 $\mathcal{F}=\{S(X,Y)\}$  of $K3$ surfaces  with $2$ complex parameters given by an affine equation in $(x,y,z)$-space:
$$
S(X,Y):     z^2=x^3 - 4 y^2(4y -5) x^2 + 20 Xy^3 x+Y y^4.
$$
We show that  a system of generators of the field of the Hilbert modular functions for $\mathbb{Q}(\sqrt{5})$ is given by the inverse of the period mapping for $\mathcal{F}$
and obtain an explicit expression of  these Hilbert modular functions given by theta constants.

We use the  following results of the Hilbert modular functions for $\mathbb{Q}(\sqrt{5})$.
Hirzebruch \cite{Hirzebruch} 
studied the Hilbert modular orbifold $\overline{(\mathbb{H}\times \mathbb{H})/\langle PSL(2,\mathcal{O}) , \tau\rangle}$,
where $\mathcal{O}=\mathbb{Z}+\displaystyle\frac{1+\sqrt{5}}{2}\mathbb{Z}$
and $\tau$ is an involution of $\mathbb{H}\times\mathbb{H}$,
 by an algebrogeometric method. 
He determined the structure of the ring of the symmetric Hilbert modular forms.
This ring is isomorphic to  the Klein 
icosahedral ring  $\mathbb{C}[\mathfrak{A},\mathfrak{B},\mathfrak{C},\mathfrak{D}]/(R(\mathfrak{A},\mathfrak{B},\mathfrak{C},\mathfrak{D})=0)$.
M\"{u}ller \cite{Muller} obtained a system of generators $\{g_2, s_6,s_{10},s_{15} \}$ of the ring of the symmetric Hilbert modular forms for $\mathbb{Q}(\sqrt{5})$ and found the relation $M(g_2,s_6,s_{10},s_{15})=0$.
These generators are given by  the  theta constants.
Then,  they are holomorphic functions on $\mathbb{H}\times\mathbb{H}$.

We  show that 
the period mapping for $\mathcal{F}$ gives a biholomorphic correspondence between the monodromy covering of $(X,Y)$-space and $\mathbb{H}\times\mathbb{H}$
and the projective monodromy group coincides with the extended Hilbert modular group $\langle PSL(2,\mathcal{O}) , \tau\rangle$.
Then, the quotient space $(\mathbb{H}\times\mathbb{H})/\langle PSL(2,\mathcal{O}),\tau \rangle$
becomes to be  the classifying space of the family $\mathcal{F}$. 
Consequently, we may regard $X$ and $Y$ as Hilbert modular functions for $\mathbb{Q}(\sqrt{5})$.
This framework  enable us to obtain 
 explicit relations between the result of \cite{Hirzebruch} and \cite{Muller}. 
Namely, we obtain an expression of the parameters $X$ and $Y$  as  quotients of    theta constants    by use of
the period mapping for
 our family $\mathcal{F}$ of $K3$ surfaces.

\begin{figure}[h]
\centering
\includegraphics{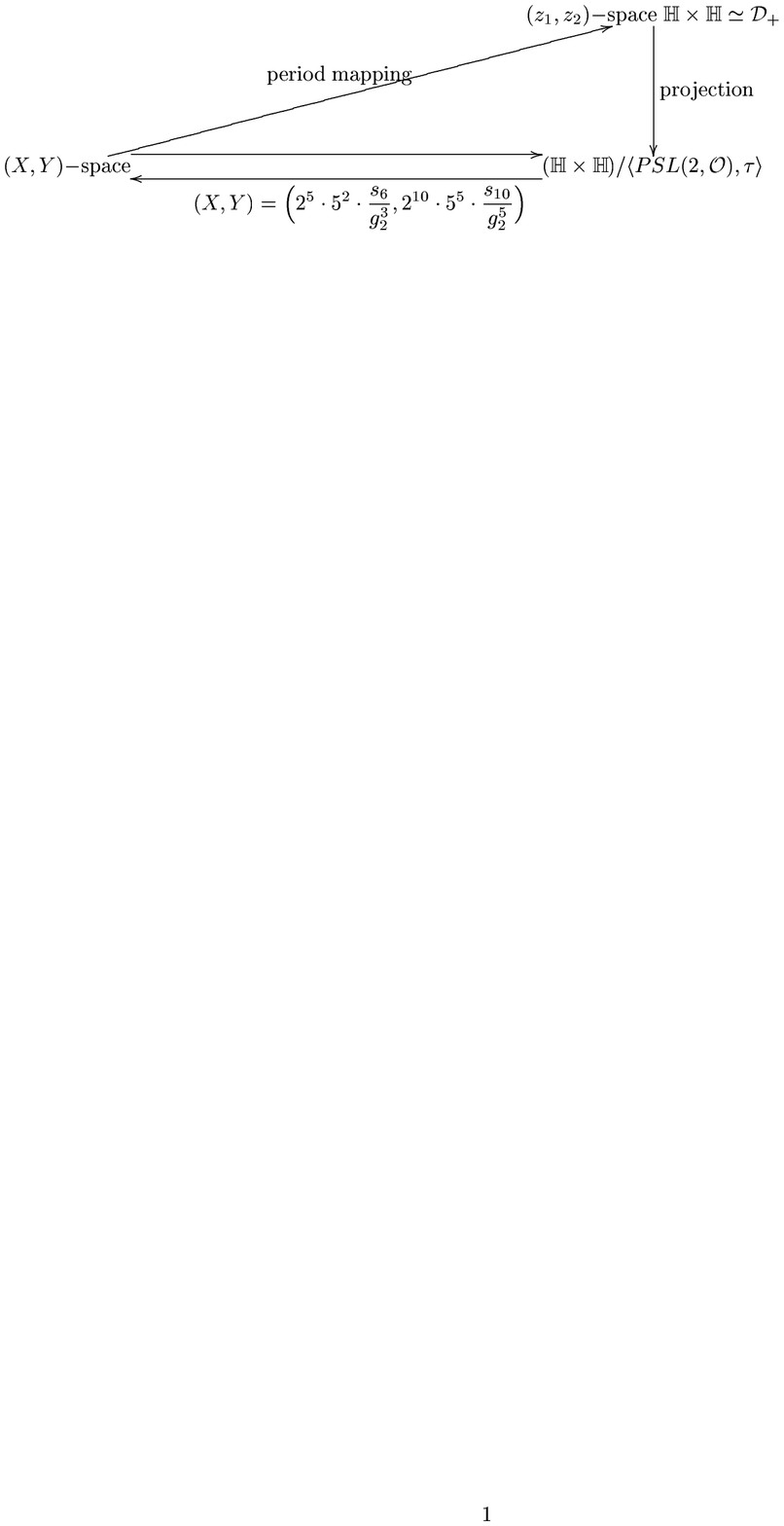}
\end{figure}

In Section 1, we give a survey of the results  of \cite{Nagano} and the properties of the Hilbert modular orbifold for $\mathbb{Q}(\sqrt{5})$.
Especially, we recall the family $\mathcal{F}_0$ of $K3$ surfaces and the period differential equation (\ref{periodUDE}) for $\mathcal{F}_0$.
A generic member of $\mathcal{F}_0$ is transformed to $S(X,Y)\in\mathcal{F}$.
The system (\ref{periodUDE}) turns out to be  the period differential equation for $\mathcal{F}$, that gives an analogy of the Gauss hypergeometric equation $\displaystyle {}_2 E_1\Big( \frac{1}{2},\frac{1}{2},1\Big)$.

In Section 2,  we study the $K3$ surface $S(X,Y)$.
First, we obtain the weighted projective space $\mathbb{P}(1,3,5)$ as a compactification of the $(X,Y)$-space  $\mathbb{C}^2$.
This
remains a parameter space for $K3$ surfaces 
except one point
(Theorem \ref{thm:compact}).
We note that, due to \cite{Hirzebruch} together with Klein \cite{Klein},
the orbifold $\overline{(\mathbb{H}\times\mathbb{H})/\langle PSL(2,\mathcal{O}),\tau \rangle}$
is isomorphic to $\mathbb{P}(1,3,5)$ as algebraic varieties.
Secondly, we define the multivalued period mapping $\mathbb{P}(1,3,5)-\{\rm one \hspace{1mm}point\} \rightarrow \mathcal{D}$ for $\mathcal{F}$, where $\mathcal{D} $ is a symmetric Hermitian space of type $IV$.
We have a modular isomorphism between $\mathbb{H}\times \mathbb{H}$ and a connected component $\mathcal{D}_+$ of $\mathcal{D}$.
Our period mapping gives an explicit isomorphism between $\mathbb{P}(1,3,5)$ and $\overline{(\mathbb{H}\times\mathbb{H})/\langle PSL(2,\mathcal{O}),\tau \rangle}$.
Then, we obtain
 the coordinates of $\mathbb{H}\times \mathbb{H}$ given by the quotients of period integrals of $S(X,Y)$:
\begin{eqnarray}\label{K3 z}
\big(z_1(X,Y),z_2(X,Y)\big)=\Bigg( -\frac{\displaystyle \int_{\Gamma_3}\omega + \frac{1-\sqrt{5}}{2}\int_{\Gamma_4}\omega}{\displaystyle \int_{\Gamma_2} \omega},-\frac{\displaystyle \int_{\Gamma_3}\omega +\frac{1+\sqrt{5}}{2}\int_{\Gamma_4} \omega}{\displaystyle \int_{\Gamma_2}\omega}\Bigg),
\end{eqnarray}
where $\Gamma_1,\cdots ,\Gamma_4$ are $2$-cycles on  $S(X,Y)\in\mathcal{F}$ given in Section 2.2.

Then, the inverse correspondence 
$
(z_1,z_2) \mapsto \big(X(z_1,z_2),Y(z_1,z_2)\big)
$
defines a pair of Hilbert modular functions for $\mathbb{Q}(\sqrt{5})$.
We obtain an expression of $X$ and $Y$ by the following way.

In Section 3,  we consider the subfamily $\mathcal{F}_X=\{S(X,0)\}$ of $K3$ surfaces.
The period mapping for $\mathcal{F}_X$ gives a correspondence 
between the  $X$-space and 
 the diagonal $\varDelta=\{(z_1,z_2)\in\mathbb{H}\times\mathbb{H}|z_1=z_2\}$.
We obtain the period differential equation for $\mathcal{F}_X$.
The solutions of this period differential equation are described in  terms of
 the solutions of the Gauss hypergeometric equation $\displaystyle {}_2E_1\Big(\frac{1}{12},\frac{5}{12},1\Big)$.
 Then, we obtain an expression of   the parameter $X$ in terms of  the elliptic $J$ function (Theorem \ref{j-functionThm}).

In Section 4, we obtain an explicit expression of the inverse of the period mapping (\ref{K3 z}) by  theta constants : 
$$
(X,Y)=\Big( 2^5 \cdot 5^2 \cdot \frac{s_6 (z_1,z_2)}{g^3_2(z_1,z_2)}, 2^{10}\cdot 5^5 \cdot\frac{s_{10} (z_1,z_2)}{g_2^{5} (z_1,z_2)}\Big)
$$
where  $g_2,s_6$ and $s_{10}$ are Hilbert modular forms given by M\"uller (Theorem \ref{main}).

Our  results in this paper are used in the forthcoming paper \cite{NaganoKummer},
in which  we shall show  a   simple and new  defining equations of the family of  Kummer surfaces for the Humbert surface of invariant $5$ and  a   geometric and intuitive  interpretation of period mappings for this family.

\section{Preliminaries}

\subsection{The family $\mathcal{F}_0$}

In \cite{Nagano}, we studied the family $\mathcal{F}_0=\{S_0(\lambda,\mu)\}$ of $K3$ surfaces defined by the equation
\begin{eqnarray} \label{S(L,M)}
S_0(\lambda,\mu): x_0 y_0 z_0^2 (x_0+y_0+z_0 +1) +\lambda x_0 y_0 z_0 +\mu=0,
\end{eqnarray}
where $(\lambda,\mu)\in \Lambda= \{(\lambda,\mu)| \lambda \mu (\lambda ^2 (4\lambda-1)^3 -2 (2+25\lambda (20\lambda-1))\mu -3125 \mu^2)\not =0)\}$.
First, we recall the results of this family.

Set 
\begin{eqnarray}\label{matrixA}
A=
\begin{pmatrix}
0 & 1&0 &0\\
1&0&0&0\\
0&0&2&1\\
0&0&1&-2
\end{pmatrix}.
\end{eqnarray}
Put
$$
\mathcal{D} = \{\xi=(\xi_1:\xi_2:\xi_3:\xi_4) \in \mathbb{P}^3(\mathbb{C})| \xi A{}^t\xi=0, \xi A {}^t \overline{\xi} >0)\}.
$$
This is a $2$-dimensional symmetric Hermitian space  of type $IV$.
Note that $\mathcal{D}$ is composed of two connected components: $\mathcal{D}=\mathcal{D}_+ \cup \mathcal{D}_-$.
We let $(1:1:-\sqrt{-1}:0) \in \mathcal{D}_+$. 
Set
$PO(A,\mathbb{Z})=\{g\in PGL(4,\mathbb{Z} ) | {}^tg A g=A \}$. It acts on $\mathcal{D}$ by ${}^t\xi \mapsto g{}^t\xi$.
Let $PO^+(A,\mathbb{Z})= \{ g\in PO(A,\mathbb{Z}) | g(\mathcal{D}_+) = \mathcal{D}_+\}$. 

In Section 2 of \cite{Nagano},
we had
the multivalued period mapping $\Phi_0:\Lambda\rightarrow \mathcal{D}_+$ for $\mathcal{F}_0$ 
 given by
\begin{align}\label{Period0}
\Phi_0 (\lambda,\mu)= \Big(\int_{\Gamma_1} \omega: \cdots:\int_{\Gamma_4}\omega \Big),
\end{align}
where 
$\omega$ is the unique holomorphic $2$-form on $S_0(\lambda,\mu)$ up to a constant factor
 and  $2$-cycles $\Gamma_1, \cdots ,\Gamma_4\in H_2(S_0(\lambda,\mu),\mathbb{Z})$ are given by this construction.

Let ${\rm NS}(S)$ be the N\'eron-Severi lattice of a $K3$ surface $S$.
The orthogonal complement ${\rm Tr}(S)={\rm NS}(S)^\bot$ in $H_2(S,\mathbb{Z})$ is called the transcendental lattice of $S$. 
 We proved
 
 \begin{thm}\label{NaganoThm} 
 {\rm (1) (\cite{Nagano} Theorem 2.2, 3.1)}
 For a generic point $(\lambda,\mu)\in\Lambda$, 
  the intersection matrix of   ${\rm NS}(S_0(\lambda,\mu))$ is given by 
 \begin{eqnarray}\label{NSE8E8}
 E_8(-1)\oplus E_8(-1) \oplus \begin{pmatrix} 2&1\\ 1&-2 \end{pmatrix}
 \end{eqnarray}
and the intersection matrix of ${\rm Tr}(S_0(\lambda,\mu))$ is given by 
\begin{eqnarray}\label{TrA}
U\oplus \begin{pmatrix} 2&1\\1&-2\end{pmatrix}=A.
\end{eqnarray}

 {\rm (2) (\cite{Nagano} Theorem 5.2)}
 The projective monodromy group of the period mapping $\Phi_0:\Lambda\rightarrow \mathcal{D}_+$ is isomorphic to $PO^+(A,\mathbb{Z})$.
  \end{thm}

Moreover,  we determined the partial differential equation in $2$ variables  $\lambda$ and $\mu$ of rank $4$
that is satisfied
by the periods
 for the family $\mathcal{F}_0$.
 We call this equation the period differential equation for $\mathcal{F}_0$. 
This equation has the singular locus $\Lambda$
(see \cite{Nagano} Theorem 4.1).

\subsection{The Hilbert modular orbifold $\overline{(\mathbb{H}\times\mathbb{H})/\langle PSL(2,\mathcal{O}),\tau \rangle}$}

Here, we recall the action of the Hilbert modular group on $\mathbb{H} \times \mathbb{H}$.
Let $\mathcal{O}$ be the ring of integers in the real quadratic field $\mathbb{Q}(\sqrt{5})$.
Set $\mathbb{H}_{\pm }=\{z\in \mathbb{C}|\pm {\rm Im}( z) >0\}.$ The Hilbert modular group $PSL(2,\mathcal{O})$ acts on 
$(\mathbb{H}_+\times\mathbb{H}_+)\cup (\mathbb{H}_-\times\mathbb{H}_-)$ by 
\begin{eqnarray*}
\begin{pmatrix}
\alpha &\beta \\
\gamma &\delta 
\end{pmatrix}
:
(z _1, z_2)\mapsto 
\Big(\frac{\alpha z_1 +\beta}{\gamma z_1 +\delta} ,\frac{\alpha ' z_2+\beta'}{\gamma' z_2+\delta' }\Big),
\end{eqnarray*}
for 
$
g = \begin{pmatrix}\alpha&\beta\\ \gamma&\delta\end{pmatrix} \in PSL(2,\mathcal{O}),
$
where $'$ means the conjugate in $\mathbb{Q}(\sqrt{5})$.
We consider the involution
$$
\tau:(z_1,z_2)\mapsto(z_2,z_1)
$$
also.

\begin{df}
If a holomorphic function $g$ on $\mathbb{H}\times\mathbb{H}$
satisfies the transformation law
$$
g\Big(\frac{az_1+b}{cz_1+d},\frac{a'z_2+b'}{c'z_2+d'} \Big)= (c z_1+d)^k (c'z_2+d')^k g(z_1,z_2)
$$ 
for any $\displaystyle \begin{pmatrix} a&b\\ c&d\end{pmatrix}\in PSL(2,\mathcal{O})$,
we call g a Hilbert modular form of weight $k$ for $\mathbb{Q}(\sqrt{5})$.
If $g(z_2,z_1)=g(z_1,z_2)$,  $g$ is called a symmetric modular form.
If $g(z_2,z_1)=-g(z_1,z_2)$, $g$ is called an alternating modular form.

If a meromorphic function $f$ on $\mathbb{H}\times \mathbb{H}$ satisfies
$$
f\Big(\frac{az_1+b}{cz_1+d},\frac{a'z_2+b'}{c'z_2+d'} \Big)= f(z_1,z_2)
$$
for any $\displaystyle \begin{pmatrix} a&b\\ c&d\end{pmatrix}\in PSL(2,\mathcal{O})$,
we call $f$  a Hilbert modular function for $\mathbb{Q}(\sqrt{5})$.
\end{df}

Set 
$$
W = \begin{pmatrix}
1&1\\
\displaystyle \frac{1-\sqrt{5}}{2} &\displaystyle \frac{1+\sqrt{5}}{2}
\end{pmatrix}.
$$ 
It holds  that
$$
A = U \oplus \begin{pmatrix} 2&1\\ 1&-2 \end{pmatrix} = U \oplus WU{}^tW .
$$ 
The correspondence 
$$
j:(z_1,z_2) \rightarrow (z_1 z_2 : -1 : z_1 :z_2) (I_2\oplus W^{-1})
$$
defines a biholomorphic mapping
$$
(\mathbb{H}_+ \times \mathbb{H}_+) \cup (\mathbb{H}_- \times \mathbb{H}_-) \rightarrow \mathcal{D}.
$$
The group $PSL(2,\mathcal{O}) $ is generated by  three elements
$$
g_1 =\begin{pmatrix} 1&1 \\ 0&1 \end{pmatrix}, \quad g_2=\begin{pmatrix} \displaystyle \vspace*{0.2cm}1 &\displaystyle \frac{1+\sqrt{5}}{2} \\ 0&1 \end{pmatrix}, \quad g_3=\begin{pmatrix} 0 &1 \\ -1& 0 \end{pmatrix}.
$$
We have an isomorphism
\begin{align*}
\begin{matrix}
\tilde{j}&: &\langle PSL(2,\mathcal{O}),\tau \rangle &\rightarrow &PO^+(A,\mathbb{Z})\\
 &;  &g&\mapsto &j\circ g\circ j^{-1}=\tilde{j}(g)=:\tilde{g}.
\end{matrix}
\end{align*}
Especially, we see
\begin{align}\label{monodromyUDE}
\begin{cases}
& \tilde{g_1}=\begin{pmatrix} 
1& -1 &2 &1\\
0 & 1 &0&0\\
0&-1&1&0 \\
0&0&0&1
\end{pmatrix},
\quad \quad\quad
\tilde{g_2}=\begin{pmatrix}
1&-1&2&1\\
0&1&0&0\\
0&-1&1&0\\
0&1&0&1
\end{pmatrix}\\
&\tilde{g_3}=\begin{pmatrix}
0&-1&0&0\\
-1&0&0&0\\
0&0&-1&-1\\
0&0&0&1
\end{pmatrix},
\quad\quad\quad\quad
\tilde{\tau}=\begin{pmatrix}
1&0&0&0\\
0&1&0&0\\
0&0&1&1\\
0&0&0&-1
\end{pmatrix}.
\end{cases}
\end{align}
So, the above  $j$ gives  a modular isomorphism
\begin{eqnarray}\label{modulariso}
j:(\mathbb{H} \times \mathbb{H}, \langle PSL(2,\mathcal{O}),\tau\rangle) \simeq (\mathcal{D}_+,PO^+(A,\mathbb{Z})).
\end{eqnarray}
Recall $\mathcal{D}=\mathcal{D}_+\cup\mathcal{D}_-$ and the period mapping $\Phi$ for $\mathcal{F}_0$.
The mapping $j^{-1} \circ \Phi : \Lambda
\rightarrow
 \mathbb{H} \times \mathbb{H}$ gives an explicit transcendental correspondence 
between $\Lambda$ and $\mathbb{H} \times \mathbb{H}$.

\vspace{5mm}

Hirzebruch \cite{Hirzebruch} studied the Hilbert modular orbifold $\overline{(\mathbb{H} \times \mathbb{H})/ \langle PSL(2,\mathcal{O}),\tau\rangle}$. Here, we survey his results.

The Klein icosahedral polynomials are
\begin{align}\label{Klein}
\begin{cases}
\vspace*{0.2cm}
\mathfrak{A}(\zeta_0:\zeta_1:\zeta_2)=\zeta_0^2 +\zeta_1 \zeta_2,\\
\vspace*{0.2cm}
\mathfrak{B}(\zeta_0:\zeta_1:\zeta_2)=8\zeta_0^4 \zeta_1 \zeta_2 -2 \zeta_0^2 \zeta_1^2 \zeta_2^2 +\zeta_1^3 \zeta_2^3 -\zeta_0 (\zeta_1^5 + \zeta_2 ^5),\\
\mathfrak{C}(\zeta_0:\zeta_1:\zeta_2)=320\zeta_0^6 \zeta_1^2 \zeta_2^2 -160 \zeta_0^4 \zeta_1^3 \zeta_2^3 +20 \zeta_0^2 \zeta_1^4 \zeta_2^4 +6 \zeta_1^5 \zeta_2^5\\
\vspace*{0.2cm}
\quad\quad\quad\quad \quad\quad\quad\quad -4\zeta_0(\zeta_1^5+\zeta_2^5) (32 \zeta_0^4-20 \zeta_0^2 \zeta_1 \zeta_2 +5\zeta_1^2 \zeta_2^2)+\zeta_1^{10}+\zeta_2^{10},\\
12\mathfrak{D}(\zeta_0:\zeta_1:\zeta_2)=(\zeta_1^5-\zeta_2^5)(-1024 \zeta_0^{10}+3840 \zeta_0^8 \zeta_1 \zeta_2 -3840 \zeta_0^6 \zeta_1^2 \zeta_2^2\\
\quad\quad\quad\quad\quad\quad\quad \quad\quad\quad\quad\quad\quad\quad\quad\quad\quad\quad+1200 \zeta_0^4 \zeta_1^3 \zeta_2^3 -100\zeta_0^2 \zeta_1^4 \zeta_2^4 +\zeta_1^5 \zeta_2^5)\\
\quad\quad\quad\quad\quad\quad\quad \quad\quad\quad\quad\quad+\zeta_0(\zeta_1^{10}-\zeta_2^{10})(352\zeta_0^4 -160 \zeta_0^2 \zeta_1 \zeta_2 +10 \zeta_1^2 \zeta_2^2)+(\zeta_1^{15}-\zeta_2^{15}).
\end{cases} 
\end{align}
We have the following relation:
\begin{eqnarray} \label{Klein}
R(\mathfrak{A},\mathfrak{B},\mathfrak{C},\mathfrak{D}):=144\mathfrak{D}^2-(-1728\mathfrak{B}^5 +720\mathfrak{A}\mathfrak{C}\mathfrak{B}^3 -80 \mathfrak{A}^2 \mathfrak{C}^2 \mathfrak{B} +64\mathfrak{A}^3(5\mathfrak{B}^2-\mathfrak{A}\mathfrak{C})^2+\mathfrak{C}^3)=0.
\end{eqnarray}
Set
\begin{eqnarray}\label{XY}
X=\frac{\mathfrak{B}}{\mathfrak{A}^3},   \quad\quad\quad Y=\frac{\mathfrak{C}}{\mathfrak{A}^5}.
\end{eqnarray}

Now, set 
$$
\Gamma(\sqrt{5})=\Big\{\begin{pmatrix} \alpha& \beta \\ \gamma & \delta \end{pmatrix} \Big| \alpha \equiv \delta \equiv 1, \beta\equiv \delta \equiv 0\quad({\rm mod}\sqrt{5})  \Big\}.
$$
We note that the group $PSL(2,\mathcal{O})/\Gamma(\sqrt{5})$ is  isomorphic to the alternating group $\mathcal{A}_5$.
Hirzebruch \cite{Hirzebruch} studied the canonical bundle of the orbifold $\overline{(\mathbb{H}\times \mathbb{H})/  \Gamma(\sqrt{5})}$  by an algebrogeometric method. 
He proved

\begin{prop} \label{KleinH} {\rm (\cite{Hirzebruch} pp.307-310)}
{\rm (1)}    The  non-singular model of $\overline{(\mathbb{H}\times \mathbb{H})/\langle \Gamma(\sqrt{5}) ,\tau\rangle}$ is $\mathbb{P}^2(\mathbb{C})=\{(\zeta_0;\zeta_1;\zeta_2)\}$ by adding six points. 
A homogeneous polynomial of degree $k$ in $\zeta_0,\zeta_1$ and $\zeta_2$ defines a modular form for $\Gamma(\sqrt{5})$ of weight $k$.

{\rm (2) }
The ring of  symmetric modular forms for  $PSL(2,\mathcal{O})$ is isomorphic to the ring
\begin{align*}
\mathbb{C}[\mathfrak{A},\mathfrak{B},\mathfrak{C},\mathfrak{D}]/(R(\mathfrak{A},\mathfrak{B},\mathfrak{C},\mathfrak{D})=0),
\end{align*}
where $R(\mathfrak{A},\mathfrak{B},\mathfrak{C},\mathfrak{D})$ is the Klein relation {\rm (\ref{Klein})}.
$\mathfrak{A} $ {\rm(}$\mathfrak{B},\mathfrak{C},\mathfrak{D}$, resp.{\rm )} gives a
 symmetric modular form for $PSL(2,\mathcal{O})$ of weight $2$ {\rm (}$6,10,15$, resp.{\rm )}.

{\rm (3)} 
There exists an alternating modular form $\mathfrak{c}$ of weight $5$ such that $\mathfrak{c}^2=\mathfrak{C}$.
 The ring of Hilbert modular forms for $PSL(2,\mathcal{O})$ is isomorphic  to the ring
\begin{align*}
\mathbb{C}[\mathfrak{A},\mathfrak{B},\mathfrak{c},\mathfrak{D}]/(R(\mathfrak{A},\mathfrak{B},\mathfrak{c}^2,\mathfrak{D})=0).
\end{align*}
\end{prop}

Let $c'\in \mathbb{C}-\{0\}$.
We consider the action $(\zeta_0,\zeta_1,\zeta_2)\mapsto (c'\zeta_0,c'\zeta_1, c'\zeta_2)$.
Because $\mathfrak{A}$ ($\mathfrak{B},\mathfrak{C}$,  resp.)
 is a homogeneous polynomial  of degree  $2$ ($6,10$, resp) in $\zeta_0,\zeta_1$ and $\zeta_2$,
 we have the action 
 $(\mathfrak{A},\mathfrak{B},\mathfrak{C})\mapsto (c'^2 \mathfrak{A},c'^6 \mathfrak{B}, c'^{10} \mathfrak{C})$.
 Therefore, we regard $(\mathfrak{A},\mathfrak{B},\mathfrak{C})$-space as the weighted projective space $\mathbb{P}(1,3,5)$.
Especially, the pair 
\begin{eqnarray}\label{weight}
\displaystyle (X,Y)=\Big( \frac{\mathfrak{B}}{\mathfrak{A}^3},\frac{\mathfrak{C}}{\mathfrak{A}^5} \Big)
\end{eqnarray}
gives a system of  affine coordinates  on 
 $\{ \mathfrak{A} \not=0 \} $.
 
By the arguments of Klein \cite{Klein}, Hirzebruch \cite{Hirzebruch} and  Kobayashi, Kushibiki and Naruki {\rm \cite{KobaNaru}}, we know the following properties of the action of $\mathcal{A}_5$ on $\overline{(\mathbb{H}\times\mathbb{H})/\langle \Gamma(\sqrt{5}),\tau \rangle} 
 =\mathbb{P}^2(\mathbb{C})=\{\zeta_0:\zeta_1:\zeta_2\}$.

\begin{prop}\label{noteH} 

{\rm (1)}
The correspondence 
$(\zeta_0:\zeta_1:\zeta_2)\mapsto \big(\mathfrak{A}(\zeta_0:\zeta_1:\zeta_2):\mathfrak{B}(\zeta_0:\zeta_1:\zeta_2):\mathfrak{C}(\zeta_0:\zeta_1:\zeta_2)\big)$
gives an identification between $\overline{\mathbb{P}^2(\mathbb{C})/\mathcal{A}_5}$ and 
$\mathbb{P}(1,3,5)$.
Then, the Hilbert modular orbifold $\overline{(\mathbb{H}\times\mathbb{H})/\langle PSL(2,\mathcal{O}),\tau \rangle}$
is identified with $\mathbb{P}(1,3,5)$.
The cusp $\overline{(\sqrt{-1}\infty,\sqrt{-1}\infty)}\in \overline{(\mathbb{H}\times\mathbb{H})/\langle PSL(2,\mathcal{O}),\tau \rangle}$ is given by the point $(\mathfrak{A}:\mathfrak{B}:\mathfrak{C})=(1:0:0)$.
So, the quotient space 
$(\mathbb{H}\times\mathbb{H})/\langle PSL(2,\mathcal{O}),\tau\rangle$
corresponds to $\mathbb{P}(1,3,5)-\{(1:0:0)\}$.

{\rm(2)} 
The divisor $\{\mathfrak{D}=0\}$ consists of fifteen  lines in $\mathbb{P}^2(\mathbb{C})$.
These fifteen lines of $\{\mathfrak{D}=0\}$ are the reflection lines of  fifteen involutions of $\mathcal{A}_5$
{\rm (}note that $\mathcal{A}_5$ is generated by three involutions{\rm )}.

{\rm (3)} The involution $\tau$ induces an involution on the orbifold  $\overline{(\mathbb{H}\times\mathbb{H})/ PSL(2,\mathcal{O})}$.
The branch locus of  the canonical projection $\overline{(\mathbb{H}\times\mathbb{H})/ PSL(2,\mathcal{O}) }
\rightarrow
\mathbb{P}(1,3,5)$
is given by $\{\mathfrak{C}=0\}$.

\end{prop}

Set
\begin{eqnarray}\label{X}
\mathfrak{X}=\{(X,Y)\in\mathbb{C}^2| Y (1728 X^5 -720 X^3 Y + 80 X Y^2 -64(5X^2-Y)^2-Y^3)\not=0\}.
\end{eqnarray}
In Section 6 of \cite{Nagano},
we obtained the birational mapping $\Lambda\rightarrow \mathfrak{X}$ given by
\begin{align}\label{(L,M)to(X,Y)}
(\lambda,\mu)\mapsto (X,Y) =\Big(\frac{25 \mu}{2 (\lambda-1/4)^3},\frac{-3125\mu^2}{(\lambda-1/4)^5} \Big).
\end{align}

\begin{thm}\label{UDEThm}
{\rm (\cite{Nagano} Theorem 6.3)}
By the correspondence {\rm (\ref{(L,M)to(X,Y)})}, the period differential equation for the family $\mathcal{F}_0=\{ S_0(\lambda,\mu)\}$ is transformed to the system of differential equations
\begin{eqnarray}
\begin{cases}\label{periodUDE}
& u_{XX}= L _1u_{XY} + A_1 u_X + B_1 u_Y +P_1 u,\\
& u_{YY}=M _1u_{XY} + C _1u_X +D_1 u_Y+Q_1 u
\end{cases}
\end{eqnarray}
with
\begin{align*}
\begin{cases}
\vspace*{0.2cm}
L_1=\displaystyle \frac{-20 (4 X^2+3 XY -4Y)}{36 X^2 -32 X -Y},\quad
M_1=\displaystyle \frac{-2(54 X^3 -50 X^2 -3 XY +2Y )}{5 Y (36 X^2 -32 X -Y)},\\
\vspace*{0.2cm}
A_1=\displaystyle \frac{-2(20 X^3 -8 XY   + 9X^2 Y + Y^2)}{XY (36X^2 -32 X - Y)},\quad
B_1=\displaystyle \frac{10Y(-8 +3X)}{X(36X^2 -32 X -Y)},\\
\vspace*{0.2cm}
C_1=\displaystyle \frac{-2(-25X^2 + 27 X^3 + 2Y -3 XY )}{5Y^2(36X^2 -32 X -Y)},\quad
D_1=\displaystyle \frac{-2(-120 X^2 +135 X^3 -2Y - 3XY)}{5XY( 36X^2 -32X -Y)}\\
\vspace*{0.2cm}
P_1=\displaystyle \frac{-2(8X-Y)}{X^2 (36X^2 -32 X -Y)},\quad
Q_1=\displaystyle \frac{-2(-10+9X)}{25XY (36X^2 - 32 X -Y)}.
\end{cases}
\end{align*}
\end{thm}

\begin{rem}\label{UDE}
In {\rm \cite{Nagano}},
we saw that {\rm (\ref{periodUDE})} is an uniformizing differential equation of the Hilbert modular orbifod $\overline{(\mathbb{H}\times\mathbb{H})/ \langle PSL(2,\mathcal{O}) ,\tau\rangle }$.
In other words, 
the solutions of {\rm (\ref{periodUDE})} define the  developing map of
the canonical projection
 $\mathbb{H}\times \mathbb{H} \rightarrow (\mathbb{H}\times \mathbb{H})/\langle PSL(2,\mathcal{O}),\tau\rangle$.
This gives an alternative proof of  {\rm Theorem \ref{NaganoThm} (2)}.
\end{rem}

\section{The period of the family $\mathcal{F}$}

\subsection{The family $\mathcal{F}$ of $K3 $ surfaces}

We obtain a new  family $\mathcal{F}$ of $K3$ surfaces with explicit defining equations 
from the family $\mathcal{F}_0=\{S_0(\lambda,\mu)\}$. 

\begin{prop}\label{PropS(X,Y)}
The family of $K3$ surfaces $\mathcal{F}_0=\{S_0(\lambda,\mu)\}$ for $(\lambda,\mu)\in \Lambda$ is transformed to the family $\mathcal{F}=\{S(X,Y)\}$ for $(X,Y)\in \mathfrak{X}$:
\begin{eqnarray}\label{S(X,Y)}
S(X,Y):     z^2=x^3 - 4 y^2(4y -5) x^2 + 20 Xy^3 x+Y y^4.
\end{eqnarray}
\end{prop}

\begin{proof}
By the  transformation  (\ref{(L,M)to(X,Y)}) and the birational transformation given by
\begin{align*}
\begin{cases}
\vspace{2mm}
&x_0 =\displaystyle\frac{Y y}{10 X x_1},\\
\vspace{2mm}
&y_0=\displaystyle\frac{4  Y^2 x_1 y_1^2}{-50 X^2 Y x_1 y_1 -5 X Y^2 y_1^2 +5 X Y z_1} ,\\
&z_0=\displaystyle - \frac{10 X Y x_1 y_1 + Y^2 y_1^2 - Y z_1}{20 X Y x_1 y_1},
\end{cases}
\end{align*}
the family $\mathcal{F}_0=\{S_0(\lambda,\mu)\}$ is transformed to the family $\mathcal{F}_1=\{S_1(X,Y)\}$ given by
$$
S_1(X,Y): z_1^2 =Y (x_1^3 - 4 y_1^2(4y_1 -5) x_1^2 + 20 Xy_1^3 x_1+Y y_1^4)
$$
over $\mathfrak{X}$.
Then, by the correspondence $(x_1,y_1,z_1) \mapsto (x,y,z)=\Big(x_1,y_1, \displaystyle\frac{1}{\sqrt{Y}} z_1\Big)$,
we have the family $\mathcal{F}=\{S(X,Y)\}$ given by (\ref{S(X,Y)}).
\end{proof}

 From (\ref{Period0}),
 we obtain the multivalued analytic period mapping
 \begin{align}\label{periodmap}
 \Phi_1:\mathfrak{X}\rightarrow \mathcal{D}_+; (X,Y)\mapsto\Big( \int_{\Gamma_1}\omega: \int_{\Gamma_2}\omega:  \int_{\Gamma_3}\omega:  \int_{\Gamma_4}\omega\Big),
 \end{align}
 where $\omega=\displaystyle \frac{dx\wedge dy}{z}$ is the unique holomorphic $2$-form on $S(X,Y)$ up to a constant factor and
 $\Gamma_1,\cdots ,\Gamma_4$ are certain $2$-cycles on $S(X,Y)$
 (this period mapping is stated in detail at the beginning of Section 2.2).
 
\begin{rem}
The correspondence
$(x_1,y_1,z_1)\mapsto (x,y,z)=(x_1,y_1,\frac{1}{\sqrt{Y}}z_1)$ in the proof of Proposition \ref{PropS(X,Y)} induces the double covering 
$\mathfrak{X}' \rightarrow \mathfrak{X}$ given by
$(X,Y')\mapsto(X,Y)=(X,Y'^2)$.
However,  $(X,Y')$ and $(X,-Y')\in\mathfrak{X}'$ define mutually isomorphic $P$-marked $K3$ surfaces  (see {\rm Definition \ref{P-marking}}).
So, we obtain the above period mapping $\Phi_1$ on $\mathfrak{X}$.
\end{rem}

Hence, from Theorem \ref{NaganoThm}, 
for a generic point $(X,Y)\in \mathfrak{X}$, the intersection matrix of the N\'eron-Severi lattice ${\rm NS}(S(X,Y))$ is given by (\ref{NSE8E8}) 
and that of the transcendental lattice  ${\rm Tr}(S(X,Y))$ is given by $A$ in (\ref{TrA}).
The projective monodromy group of $\Phi_1$ is isomorphic to $PO^+(A,\mathbb{Z})$.
From Theorem \ref{UDEThm}, 
the period differential equation for the family $\mathcal{F}=\{S(X,Y)\}$ is given by {\rm (\ref{periodUDE})}.

\begin{prop}
Under the correspondence {\rm (\ref{weight})}, the surface $S(X,Y)$ is birationally equivalent  to
\begin{eqnarray}\label{S(a_0;a_1;a_2)}
S(\mathfrak{A}:\mathfrak{B}:\mathfrak{C}):z^2=x^3-4(4y^3 -5\mathfrak{A}y^2)x^2+20\mathfrak{B} y^3 x +\mathfrak{C}y^4.
\end{eqnarray}
\end{prop}

\begin{proof}
Putting $\displaystyle X=\frac{\mathfrak{B}}{\mathfrak{A}^3},Y=\frac{\mathfrak{C}}{\mathfrak{A}^5}$ to
(\ref{S(X,Y)}), we have
$$
 \mathfrak{A}^{5} z^2 = \mathfrak{A}^{5} x^3 + (20y^2 -16 y^3)\mathfrak{A}^{5}x^2+20 \mathfrak{A}^2 \mathfrak{B} y^3 x +\mathfrak{C}y^4.
$$
Then, 
by the correspondence
$$
x\mapsto \frac{x}{\mathfrak{A}^3},\quad\quad y\mapsto\frac{y}{\mathfrak{A}},\quad\quad z\mapsto\frac{z}{\sqrt{\mathfrak{A}^9}},
$$
we obtain (\ref{S(a_0;a_1;a_2)}).
\end{proof}

\begin{rem}
For two surfaces
\begin{align*}
\begin{cases}
&S(\mathfrak{A}:\mathfrak{B}:\mathfrak{C}):z^2=x^3-4(4y^3 -5\mathfrak{A}y^2)x^2+20\mathfrak{B} y^3 x +\mathfrak{C}y^4,\\
&S(k^2 \mathfrak{A}:k^6\mathfrak{B}:k^{10}\mathfrak{C}):z^2=x^3-4(4y^3 -5k^2\mathfrak{A}y^2)x^2+20k^6\mathfrak{B} y^3 x +k^{10}\mathfrak{C}y^4,
\end{cases}
\end{align*}
we have an isomorphism 
$
S(\mathfrak{A}:\mathfrak{B}:\mathfrak{C}) \rightarrow S(k^2 \mathfrak{A}:k^6\mathfrak{B}:k^{10}\mathfrak{C})$
given by 
$
(x,y,z)  \mapsto  (k^6 x, k^2 y , k^{9} z)
$
as elliptic surfaces.
Therefore, $(\mathfrak{A}:\mathfrak{B}:\mathfrak{C})\in\mathbb{P}(1:3:5)$ gives an isomorphism class of these elliptic $K3$ surfaces.
\end{rem}

We set $K_1=\{Y=0\}$ and $K_2=\{1728 X^5 -720 X^3 Y + 80 X Y^2 -64(5X^2-Y)^2-Y^3=0 \}$.

\begin{thm}\label{thm:compact}
The $(\mathfrak{A}:\mathfrak{B}:\mathfrak{C})$-space
$\mathbb{P}(1,3,5)$ 
 gives a compactification  of the parameter space $\mathfrak{X}$ of  the family $\mathcal{F}=\{S(X,Y)\}$ of $K3$ surfaces  given by {\rm (\ref{S(X,Y)})}.
 Namely, if $(1:0:0) \not=  (\mathfrak{A}:\mathfrak{B}:\mathfrak{C})\in \mathbb{P}(1,3,5)$, then the corresponding  surface $S(\mathfrak{A}:\mathfrak{B}:\mathfrak{C})$ is a $K3$ surface. On the other hand, $S(1:0:0)$ is a rational surface.  
\end{thm}

\begin{proof}
First, we prove the case $\mathfrak{A}\not=0$. In this case, we consider $S(X,Y)$ in $(\ref{S(X,Y)})$.
We have the Kodaira normal form of (\ref{S(X,Y)}):
\begin{eqnarray}\label{Kodaira1}
z_1^2=x_1^3 -g_2(y) x -g_3(y)  \quad\quad(y\not=\infty),
\end{eqnarray}
with
\begin{align*}
\begin{cases}\vspace{2mm}
&\displaystyle  g_2 (y)= -\Big(20X y^3-\frac{16}{3} y^4 (4y-5)^2\Big)\\
&\displaystyle g_3(y)=-\Big(Y y^4+\frac{80}{3} y^5 (4y-5) X -\frac{128}{27} y^6 (4y-5)^3 \Big),
\end{cases}
\end{align*}
and 
\begin{eqnarray}\label{Kodaira2}
z_2^2=x_2^3 -h_2(y_1) x_2 -h_3(y_1)     \quad\quad (y\not=0),
\end{eqnarray}
with 
\begin{align*}
\begin{cases}\vspace{2mm}
&\displaystyle h_2(y_1) =-\Big(20 X y_1^5-\frac{256}{3} y_1^2+\frac{640}{3}y_1^3 -\frac{400}{3}y_1^4\Big),\\
&\displaystyle h_3(y_1)=-\Big(Y y_1^8 +\frac{320 }{3} X y_1^6 -\frac{400}{3} X y_1^7 -\frac{8192}{27} y_1^3 +\frac{10240}{9} y_1^4- \frac{12800}{9} y_1^5 +\frac{16000}{27} y_1^6  \Big),
\end{cases}
\end{align*}
where $\displaystyle y_1=\frac{1}{y}$.
The discriminant $D_0$ ($D_\infty$, resp.) of the right  hand side of (\ref{Kodaira1}) ((\ref{Kodaira2}), resp.) is given by
\begin{align*}
\begin{cases}
 &D_0=y^8(27 Y^2 +32000 X^3 y -7200 XY y -160000 X^2 y^2 +32000 Y y^2 \\
&\hspace{2cm}+5760 X Y y^2 +256000 X^2 y^3 -76800 Y y^3 - 102400 X^2 y^4 +61440 Y y^4 -16384 Y y^5),\\
&D_\infty = y_1^{11} (-16384 Y -102400 X^2 y_1 +61440 Y y_1 +256000 X^2 y_1^2 -76800 Y y_1^2\\
&\hspace{2cm} -160000 X^2 y_1^3 +32000 Y y_1^3 +5760 XY y_1^3 +32000 X^3 y_1^4 -7200 XY y_1^4 +27 Y^2 y_1^5).
\end{cases}
\end{align*}

If $(X,Y)\in \mathfrak{X}$, then we have 
$$
{\rm ord}_y(D_0) =8, \quad\quad {\rm ord}_y (g_2) =3 , \quad\quad {\rm ord}_y (g_3) =4,
$$
so $\pi^{-1}(0)$ is the singular fibre of type $IV^*$ (for detail, see \cite{Kodaira} or \cite{Shiga}). 
Similarly, we have
$$
{\rm ord}_y (D_\infty) =11, \quad\quad {\rm ord}_y (h_2) =2 , \quad\quad {\rm ord}_y (h_3) =3,
$$
so $\pi^{-1}(\infty)= I_5^*$. 
We have other 5 singular fibres of type  $I_1$.
Therefore, for  $(X,Y)\in \mathfrak{X}$, $S(X,Y)$ is an elliptic $K3$ surface whose singular fibres are of type  $IV^* + 5 I_1 +I_5^* $.

By the same way, we know  the structure of the elliptic surface $S(X,Y)$ for $(X,Y)\not\in\mathfrak{X}$.
If $X\not=0 $ and $Y=0$ (namely, $(X,Y)\in K_1-\{(0,0)\}$), then $S(X,0)$ is an elliptic $K3$ surface with the singular fibres  of type $III^* + 3I_1 + I_6^*$.
If $(X,Y)\in K_2-\{(0,0)\}$, $S(X,Y)$ is an elliptic K3 surface with the singular fibres of type $IV^* +3 I_ 1+ I_2 +I_5^*$.
However,  we see easily that $S(0,0)$  is not  a $K3$  surface, but a rational surface.

Next, we consider the case $\mathfrak{A}=0$.
In this case,  note that $(\mathfrak{B},\mathfrak{C})\not=(0,0)$.
We have the equation of $S(0:\mathfrak{B}:\mathfrak{C})$:
$
z^2 = x^3 -16y^3 x^2+20\mathfrak{B} y^3 x +\mathfrak{C} y^4.
$
On $\{\mathfrak{A}=0\}\subset \mathbb{P}(1,3,5)$, we use the parameter $\displaystyle l=\frac{\mathfrak{C}^3}{\mathfrak{B}^5}$.
By the correspondence $\displaystyle x= \frac{\mathfrak{C}^3}{\mathfrak{B}^4}x', y=\frac{\mathfrak{C}^2}{\mathfrak{B}^3} y'$ and $\displaystyle z=\frac{\sqrt{\mathfrak{C}^9}}{\mathfrak{B}^6}z'$, we have
$$
S(l): z'^2=x'^3 - 16 l y'^3 x'^2+20 y'^3 x'+y'^4.
$$
The discriminant of the right  hand side is given by
$
y'^8 (27 +32000 y' +5760 l y'^2 -102400l^2 y'^4 -16384 l^3 y'^5).
$
From this, we can see  that $S(l)$ is an elliptic $K3$ surface with the singular fibres  of type $IV^*+5I_1 +I_5^*$.
\end{proof}

Hence, we obtain the extended family $\{S(\mathfrak{A}:\mathfrak{B}:\mathfrak{C})| (\mathfrak{A}:\mathfrak{B}:\mathfrak{C})\in \mathbb{P}(1,3,5)-\{(1:0:0)\}\}$
of $K3$ surfaces.
For simplicity, 
let $\mathcal{F} $ denotes this extended family.

\subsection{The extension $\Phi$ of the period mapping $\Phi_1$}

Set $c_0=(1:0:0)\in\mathbb{P}(1,3,5)$.
In this subsection, we extend the  period mapping $\Phi_1:\mathfrak{X}\rightarrow \mathcal{D}_+$ in (\ref{periodmap}) to $\Phi:\mathbb{P}(1,3,5)-\{c_0\}\rightarrow \mathcal{D}_+$.

First, we recall the S-marking on $\mathfrak{X}$.
According to Theorem \ref{thm:compact} and its proof,
we have the elliptic $K3$ surface
$$
\pi_{(\mathfrak{A}:\mathfrak{B}:\mathfrak{C})}: S(\mathfrak{A}:\mathfrak{B}:\mathfrak{C})\rightarrow \mathbb{P}^1(\mathbb{C})=(y{\rm -sphere})
$$
for any $(\mathfrak{A}:\mathfrak{B}:\mathfrak{C})\in\mathbb{P}(1,3,5)-\{c_0\}$.

Take a generic point $(X_0,Y_0)\in\mathfrak{X}$. 
 The elliptic $K3$ surface $\check{S}=S(X_0,Y_0)$ given by (\ref{Kodaira1}) and (\ref{Kodaira2}) has the singular fibres of type $IV^*+5I_1+I_5^*$.
Let $F$  be a general fibre of this elliptic fibration and 
 $O$ be the zero of the Mordell-Weil group of sections.
 We have  two irreducible components of the divisor $C$ given by $\{x=0, z^2=Y y^4\}$. 
 We  take  the section  $R$   given by  $y\mapsto (x,y,z)=(0,y,\sqrt{Y}y^2)$. 
 This gives a component of the divisor $C$.
Let  us consider the irreducible decomposition  $\displaystyle \bigcup_{j=0}^6 a_j$ ($\displaystyle \bigcup_{j=0}^9 b_j$, resp.) of the singular fibre $\pi^{-1}_{(X,Y)}(0)$ ($\pi^{-1}_{(X,Y)}(\infty)$, resp.) of type $IV^*$ ($I_5^*$, resp.).
These  curves are illustrated 
 in Figure 1. 
Note that    $a_0\cap O\not=\phi$, $b_0\cap O\not=\phi$,  
 $a_6\cap R\not=\phi$ and $b_9\cap R\not=\phi$.

\begin{figure}
\center
\includegraphics{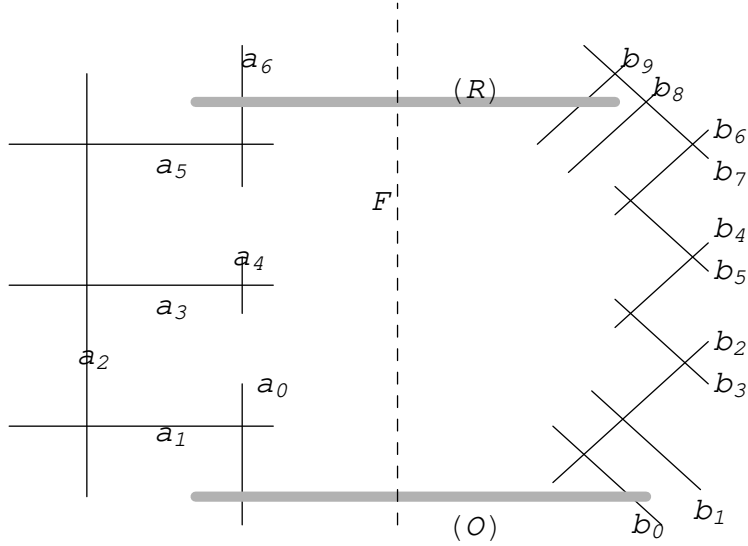} 
\caption{The elliptic fibration given by (\ref{S(a_0;a_1;a_2)}).}
\end{figure}

We set $\Gamma_5=F,\Gamma_6=O,\Gamma_7=R,$
 $\Gamma_{8+k}=a_{k+1}$ $(0\leq k\leq 5),$ $\Gamma_{14+l}=b_{l+1}$ $ (0\leq l \leq 8)$.
We have the lattice $\check{L}=\langle \Gamma_5,\cdots,\Gamma_{22}  \rangle_\mathbb{Z} \subset H_2(\check{S},\mathbb{Z})$.
We can check that $|{\rm det}(\check{L})|=5$. 
Hence, we have 
$$
\check{L}={\rm NS}(\check{S}).
$$
Since $\check{L}$ is a primitive lattice, there exists $\Gamma_1,\cdots,\Gamma_4\in H_2(\check{S},\mathbb{Z})$ such that $\langle \Gamma_1,\cdots,\Gamma_4,\Gamma_5,\cdots,\Gamma_{22} \rangle_{\mathbb{Z}}=H_2(\check{S},\mathbb{Z})$.
Let $\{\Gamma_1^*,\cdots,\Gamma_{22}^*\}$ be the dual basis of $\{\Gamma_1,\cdots,\Gamma_{22}\}$ in $H_2(\check{S},\mathbb{Z})$.
Then,  we see that $\langle \Gamma_1^*,\cdots,\Gamma_4^* \rangle_{\mathbb{Z}}$ is the transcendental lattice. 
We may  assume  that its intersection matrix is
\begin{eqnarray}\label{TrMat}
(\Gamma_j^*\cdot \Gamma_k^*)_{1\leq j,k\leq 4}=A
\end{eqnarray}
 where $A$ is given by (\ref{matrixA}).
We define the period of $\check{S}$  by 
$$
\Phi_1(X_0,Y_0)=\Big(\int_{\Gamma_1}\omega:\cdots:\int_{\Gamma_4}\omega \Big).
$$ 
Take a small connected neighborhood $V_0$ of 
$(X_0,Y_0)$ in 
$\mathfrak{X}$ so that we have a local topological trivialization
\begin{eqnarray}\label{Ltop}
\tau:\{S(p)|p\in V_0\} \rightarrow \check{S}\times V_0.
\end{eqnarray}
Let $\varpi:\check{S}\times V_0\rightarrow \check{S}$ be the canonical projection.
Set $r=\varpi\circ\tau$. Then,
$$
r_p' = r|_{S(p)}
$$
gives a $\mathcal{C}^\infty$-isomorphism of surfaces.
For any $p \in V_0$, we have an isometry $\psi_p:H_2(S(p),\mathbb{Z})\rightarrow H_2(\check{S},\mathbb{Z})$ given by 
$$
\psi_p={r_p}'_{*}.
$$  
We call this isometry the S-marking on $V_0$.
By an analytic continuation along an arc $\alpha\subset \mathfrak{X}$, we define the S-marking on $\mathfrak{X}$. This  depends on the choice of $\alpha$.
The S-mariking  preserves the N\'{e}ron-Severi lattice.
We define the period mapping $\Phi_1:\mathfrak{X}\rightarrow \mathcal{D}_+$ by
$$
p\mapsto \Big( \int_{\psi^{-1}_p(\Gamma_1)}\omega:\cdots:\int_{\psi^{-1}_p(\Gamma_4)}\omega\Big).
$$
This is equal to  the period mapping in (\ref{periodmap}).

Here, we recall the P-marking for $K3$ surfaces, which is defined in \cite{Nagano} Section 5.

\begin{df}\label{P-marking}
Let $S$ be an algebraic $K3$ surface.
An isometry
$$
\psi:H_2(S,\mathbb{Z})\rightarrow H_2(\check{S},\mathbb{Z})
$$ 
is called the P-marking if
\par
{\rm (i)} $\psi ^{-1} ({{\rm NS}(\check{S})})\subset {\rm NS} (S)$,
\par
{\rm (ii)} $\psi^{-1} (F),\psi^{-1} (O), \psi^{-1} (R),\psi^{-1}(a_j)$  $(1\leq j\leq6)$ and $\psi^{-1}(b_j)$  $(1\leq j\leq 9)$ 
are all effective divisors,
\par
{\rm (iii)}  $(\psi^{-1} (F) \cdot C)\geq 0$ for any effective class $C$.
Namely, $\psi^{-1} (F)$ is nef.
\begin{flushleft}
A pair $(S,\psi)$ is called  a P-marked $K3$ surface. 
\end{flushleft}
\end{df}

\begin{df}\label{df:iso,equi}
Two P-marked $K3$ surfaces $(S_1,\psi_{1} )$ and $(S_2,\psi_{2})$ are said to be 
isomorphic if there is a biholomorphic mapping $f:S_1\rightarrow S_2$ with 
$$
\psi_{2} \circ f_{\ast} \circ \psi_{1}^{-1} ={\rm id}_{H_2({\check{S}},\mathbb{Z})}.
$$
Two P-marked $K3$ surfaces $(S_1,\psi_{1} )$ and $(S_2,\psi_{2})$ are said to be 
equivalent if there is a biholomorphic mapping $f:S_1\rightarrow S_2$ with 
$$
(\psi_{2} \circ f_{\ast} \circ \psi_{1}^{-1})|_{ {\rm NS}(\check{S})}  ={\rm id} _{{\rm NS} (\check{S})}.
$$
\end{df}

\begin{rem}
 The other connected component $R'$ of the divisor $C$ given by the section $y\mapsto(x,y,-\sqrt{Y}y^2)$ intersects $a_4$ ($b_8$, resp.) at $y=0$ ($y=\infty$, resp.).
Letting  $q$ be the involution of $S(X,Y)$ given by $(x,y,z)\mapsto (x,y,-z)$, 
we have
$q_*(R')=R,$
$q_*(a_4)=a_6$, $q_*(a_3)=a_5$ and $q_*(b_8)=b_9$.
Then, we can see that $P$-marked $K3$ surfaces $(\check{S},id)$ and $(\check{S},q_*)$ are isomorphic by  $q$. 
This shows that  our argument does not depend on the choice of the curves $R $ or $R'$.
\end{rem}

The period of a P-marked $K3$ surface $(S,\psi)$ is given by
\begin{eqnarray}\label{Periodmap}
\tilde{\Phi'}(S,\psi)=
\Big( \int _{\psi^{-1}(\Gamma_1)}\omega :\cdots :\int_{\psi^{-1}(\Gamma_4)}\omega\Big).
\end{eqnarray} 
It is a point in $\mathcal{D}$.
Let $\mathbb{X}$ be the isomorphism classes of P-marked $K3$ surfaces and let 
$$
[\mathbb{X}]=\mathbb{X}/(P{\rm -marked\hspace{1mm} equivalence}).
$$
By the Torelli theorem for $K3$ surfaces, the period mapping
$\tilde{\Phi'}:\mathbb{X}\rightarrow \mathcal{D}$  
for
 P-marked $K3$ surfaces defined by (\ref{Periodmap}) gives an identification between $\mathbb{X}$ and $\mathcal{D}$.
Moreover, a P-marked $K3$ surface $(S_1,\psi_1)$ is equivalent to a P-marked $K3$ surface $(S_2,\psi_2)$ if and only if 
$$
\tilde{\Phi'}(S_1,\psi_1)=g\circ \tilde{\Phi'}(S_2,\psi_2)
$$
for some $g\in PO(A,\mathbb{Z})$ (see \cite{Nagano} Lemma 5.1).
Therefore, we identify $[\mathbb{X}]$ with
\begin{eqnarray}\label{markedsp}
\mathcal{D}/PO(A,\mathbb{Z})=\mathcal{D}_+/PO^+(A,\mathbb{Z})\simeq (\mathbb{H}\times\mathbb{H})/\langle PSL(2,\mathcal{O}),\tau\rangle.
\end{eqnarray}
Recall that the above isomorphism is given by the modular isomorphism $j$ in (\ref{modulariso}). 
 
  We note that $\mathfrak{X}$ is embedded in $[\mathbb{X}]$ (see \cite{Nagano} Remark 5.3).
Then,   an S-marked $K3$ surface is a P-marked $K3$ surface and 
the period mapping for  P-marked $K3$ surfaces
 is an extension of the period mapping  for  S-marked $K3$ surfaces.
From $\tilde{\Phi'}:\mathbb{X}\rightarrow \mathcal{D}$, we obtain a multivalued mapping  $\Phi':[\mathbb{X}]\rightarrow \mathcal{D}_+$.
We have 
\begin{eqnarray}\label{rest1}
\Phi'|_\mathfrak{X}=\Phi_1,
\end{eqnarray}
where $\Phi_1$ is the period mapping  in (\ref{periodmap}) for S-marked $K3$ surfaces.

\vspace{5mm}

Now, we extend the period mapping $\Phi_1: \mathfrak{X}\rightarrow \mathcal{D}_+$ in (\ref{periodmap}) to
$\Phi:\mathbb{P}(1,3,5)-\{c_0\}\rightarrow \mathcal{D}_+$.
We recall that $(\mathbb{P}(1,3,5)-\{c_0\})-\mathfrak{X}=(K_1\cup K_2\cup \{\mathfrak{A}=0\})-\{c_0\}$.

First, since the local topological trivialization on $\mathfrak{X}$ in (\ref{Ltop}) is naturally  extended to $\{\mathfrak{A}=0\}$, there exist S-markings on $\{\mathfrak{A}=0\}$ and the period mapping (\ref{periodmap}) on $\mathfrak{X}$ is extended to  $\mathbb{P}(1,3,5)-(K_1\cup K_2 \cup \{c_0\})\rightarrow \mathcal{D}_+$.

Let us recall that the projective monodromy group of $\Phi_1$ is isomorphic to $PO^+(A,\mathbb{Z})$.
According to (\ref{markedsp}) and Proposition \ref{noteH} (3) (Proposition \ref{noteH} (2), resp.), the local monodromy of the period mapping $\Phi_1$ in (\ref{periodmap}) around $K_1$  ($K_2$, resp.) is locally finite. 
Hence,  the period mapping $\mathbb{P}(1,3,5)-(K_1\cup K_2 \cup \{c_0\})\rightarrow \mathcal{D}_+$ can be extended to  $\mathbb{P}(1,3,5)-\{c_0\} \rightarrow\mathcal{D}_+$.
We note that this extension is assured by Theorem (9.5) in Griffiths \cite{Griffiths}.

Therefore, we have the extended period mapping
\begin{align}\label{OurPeriod} 
\Phi:\mathbb{P}(1,3,5)-\{c_0\}\rightarrow \mathcal{D}_+
\end{align}
with 
\begin{eqnarray}\label{rest2}
\Phi|_\mathfrak{X} =\Phi_1.
\end{eqnarray}

Since we have (\ref{markedsp}) and Proposition \ref{noteH} (1), 
the P-marked equivalence classes $[\mathbb{X}]$ is identified with $\mathbb{P}(1,3,5)-\{c_0\}$.
 Because we have (\ref{rest1}), (\ref{rest2})
 and
 $\mathfrak{X}$ is a Zariski open set in $\mathbb{P}(1,3,5)-\{c_0\}$,
   $\Phi$ in (\ref{OurPeriod}) is 
 equal to the period mapping $\Phi'$ on  $[\mathbb{X}]$.

Let $[\Phi (p)]\in \mathcal{D}_+/PO^+(A,\mathbb{Z})$ be the  equivalence class of $\Phi(p)\in \mathcal{D}_+$. 
From the above argument, we have the following proposition.

\begin{prop}\label{markingprop}
The period mapping  $\Phi':[\mathbb{X}]\rightarrow \mathcal{D}_+$  for P-marked $K3$ surfaces  is given by the period mapping $\Phi$ in {\rm (\ref{OurPeriod})} for the family $\mathcal{F}=\{S(p)|p\in\mathbb{P}(1,3,5)-\{c_0\}\}$ of $K3$ surfaces.
This is an extension of the period mapping in {\rm (\ref{periodmap})} for S-marked $K3$ surfaces.
Especially, if $[\Phi(p_1)]=[\Phi(p_2)]$ in $\mathcal{D}_+/PO^+(A,\mathbb{Z})$, then $p_1=p_2$.
\end{prop}

For $p\in \mathbb{P}(1,3,5)-\{c_0\}$,
let
$$
\psi_p: H_2(S(p),\mathbb{Z})\rightarrow H_2(\check{S},\mathbb{Z})
$$ 
be a P-marking naturally induced by the above proposition.
The period of $S(p)$ is given by
\begin{eqnarray}\label{EPeriod}
\Phi(p)=\Big( \int_{\psi^{-1}_p(\Gamma_1)}\omega:  \int_{\psi^{-1}_p(\Gamma_2)}\omega:  \int_{\psi^{-1}_p(\Gamma_3)}\omega:   \int_{\psi^{-1}_p(\Gamma_4)}\omega\Big).
\end{eqnarray}

According to  Remark \ref{UDE},
  the multivalued analytic mapping 
$
(j^{-1}\circ \Phi)|_\mathfrak{X} :\mathfrak{X}\rightarrow \mathbb{H}\times\mathbb{H}
$ 
 gives  a developing map of the canonical projection $\Pi:\mathbb{H}\times\mathbb{H}\rightarrow (\mathbb{H}\times\mathbb{H})/\langle PSL(2,\mathcal{O}),\tau \rangle$.
Hence, by Proposition \ref{markingprop}, $(j^{-1}\circ \Phi) |_\mathfrak{X}$ is extended to the analytic mapping 
$$j^{-1}\circ\Phi: \mathbb{P}(1,3,5)-\{c_0\}\rightarrow \mathbb{H}\times \mathbb{H}.$$
This gives a developing map of $\Pi$.

\begin{rem}
 Sato {\rm \cite{Sato}} showed that
the system of differential equations on $\mathfrak{X}$
\begin{eqnarray*}
\begin{cases}
& u_{XX}= L u_{XY} + A u_X + B u_Y +P u,\\
& u_{YY}=M u_{XY} + C u_X +D u_Y+Q u
\end{cases}
\end{eqnarray*}
with $\displaystyle L=\displaystyle \frac{-20 (4 X^2+3 XY -4Y)}{36 X^2 -32 X -Y},
M=\displaystyle \frac{-2(54 X^3 -50 X^2 -3 XY +2Y )}{5 Y (36 X^2 -32 X -Y)}$ 
is an uniformizing differential equation of $\overline{(\mathbb{H}\times \mathbb{H})/\langle PSL(2,\mathcal{O}),\tau\rangle}$.
Namely, taking  linearly independent solutions $y_0,y_1,y_2$ and $y_3$,
the mapping $p\mapsto (y_0(p):\cdots:y_3(p))$ gives a developing map $\mathfrak{X}\rightarrow \mathcal{D}_+$.
Of course, our equation  {\rm (\ref{periodUDE})} is also an unifomizing differential equation in this sense.
But, note that we do not know whether we can extend it to the singular locus 
applying the  theory of the uniformizing differential equations.
Since 
we regard $\mathbb{P}(1,3,5)-\{c_0\}$ as the parameter space of $\mathcal{F}$
and  $p\mapsto (y_0(p):\cdots:y_3(p))$ is the period mapping for $\mathcal{F}$,
we obtain the  extension   of    the solutions  of {\rm (\ref{periodUDE})} to the singular locus.
\end{rem}

Hence, we obtain the following theorem.

\begin{thm}\label{thmC}
The multivalued mapping  $j^{-1}\circ\Phi: \mathbb{P}(1,3,5)-\{c_0\}\rightarrow \mathbb{H}\times \mathbb{H}$
gives the developing map of $\Pi$.
Namely, the inverse mapping of $\Pi: \mathbb{H}\times\mathbb{H}\rightarrow (\mathbb{H}\times\mathbb{H})/\langle PSL(2,\mathcal{O}),\tau\rangle$ is given by $ j^{-1}\circ\Phi$
through the identification $(\mathbb{H}\times \mathbb{H})/\langle PSL(2,\mathcal{O}),\tau \rangle \simeq \mathbb{P}(1,3,5)-\{c_0\}$ given by {\rm Proposition} {\rm \ref{noteH}} {\rm (1)}. 
\end{thm}

Let $\varDelta$ be the diagonal:
$$
\varDelta =\{(z_1,z_2)\in \mathbb{H}\times \mathbb{H}|z_1=z_2\}.
$$ 
From the above theorem and Proposition \ref{noteH} (3), we have

\begin{cor}\label{Cor Delta}
It holds that
$$
\Pi (\varDelta)=\{(\mathfrak{A}:\mathfrak{B}:0)\}-\{c_0\}
$$
through the identification $(\mathbb{H}\times \mathbb{H})/\langle PSL(2,\mathcal{O}),\tau \rangle \simeq \mathbb{P}(1,3,5)-\{c_0\}$ given by {\rm Proposition} {\rm \ref{noteH}} {\rm (1)}. 
\end{cor}

Due to  Theorem \ref{thmC},
we obtain the system of coordinates $(z_1,z_2)$ of $\mathbb{H}\times\mathbb{H}$ coming from the multivalued period mapping   (\ref{EPeriod})  for the family of $K3 $ surfaces $\{S(p)\}$:
\begin{align}\label{upperhalf}
\big(z_1(p),z_2(p)\big)=\Bigg( -\frac{\displaystyle \int_{\Gamma_3}\omega + \frac{1-\sqrt{5}}{2}\int_{\Gamma_4}\omega}{\displaystyle \int_{\Gamma_2} \omega},-\frac{\displaystyle \int_{\Gamma_3}\omega +\frac{1+\sqrt{5}}{2}\int_{\Gamma_4} \omega}{\displaystyle \int_{\Gamma_2}\omega}\Bigg).
\end{align}
Here, 
for simplicity,
let $\Gamma_j$ denotes the 2-cycle $\psi^{-1}_p(\Gamma_j)$ on $S(p)$ for  $j\in \{1,2,3,4\}$.

According to Proposition \ref{noteH} (1), by adding one cusp,
we have the compactification $\overline{(\mathbb{H}\times\mathbb{H})/\langle PSL(2,\mathcal{O}),\tau\rangle}$.
Then, putting $\Pi\circ j^{-1}\circ\Phi (c_0)=\overline{(\sqrt{-1}\infty,\sqrt{-1}\infty)}$,
 we obtain an extended  mapping 
\begin{eqnarray}\label{Qmap}
\Pi\circ j^{-1}\circ \Phi :\mathbb{P}(1,3,5)\rightarrow \overline{(\mathbb{H}\times\mathbb{
H})/\langle PSL(2,\mathcal{O}),\tau\rangle},
\end{eqnarray}
where $\overline{(\sqrt{-1}\infty,\sqrt{-1}\infty)}$ stands for the $\langle PSL(2,\mathcal{O}),\tau\rangle$ orbit of $(\sqrt{-1}\infty,\sqrt{-1}\infty)$.

\section{The family $\mathcal{F}_X$ and the period differential equation}

In this section, we consider 
the family $\mathcal{F}_X=\{S(X,0)\}$ 
and the diagonal 
$
\varDelta =\{(z_1,z_2)\in\mathbb{H}\times\mathbb{H} | z_1 =z_2  \}.
$

\subsection{The family $\mathcal{F}_X$}

In Section 2, we had the $K3$ surfaces $S(\mathfrak{A}:\mathfrak{B}:\mathfrak{C})$
for $(\mathfrak{A}:\mathfrak{B}:\mathfrak{C})\in \mathbb{P}(1,3,5)-\{c_0\}$ 
and the period mapping (\ref{EPeriod}).
Restricting 
them to $\{\mathfrak{C}=0\}$,
 we obtain the family $\{S(\mathfrak{A}:\mathfrak{B}:0)|(\mathfrak{A}:\mathfrak{B}:0)\not=c_0\}$ of $K3 $ surfaces with
$
S(\mathfrak{A}:\mathfrak{B}:0):z^2 =x^3 -4 y^2 (4y-5\mathfrak{A}) x^2 +20 \mathfrak{B} y^3 x.
$
Then, we have the family $\mathcal{F}_X=\{S(X,0)\}$ of $K3$ surfaces with
$$
S(X,0):z^2 =x^3 -4 y^2 (4y-5) x^2 +20 X y^3 x,
$$
where $X \Big(=\displaystyle \frac{\mathfrak{B}}{\mathfrak{A}^3}\Big) \in \mathbb{P}^1(\mathbb{C})-\{0\}.$
In this section,
we consider the family $\mathcal{F}_X$ and the period mapping  for $\mathcal{F}_X$.

Set $\Sigma=\big(X{\rm-sphere}\hspace{1mm} \mathbb{P}^1(\mathbb{C})\big)-\Big\{0,\displaystyle \frac{25}{27},\infty\Big\}$.
Because we have Propsition \ref{markingprop},
we can prove the following theorem
for the subfamily $\displaystyle\mathcal{F}_X'=\{ S(X,0) | X\in\Sigma \}$
as in \cite{Nagano}.

\begin{thm}\label{diagthm}

{\rm (1)} For a generic point $X\in \Sigma$, 
the intersection matrix  of the 
N\'{e}ron-Severi lattice ${\rm NS}(S(X,0))$ is given by  
$$E_8(-1)\oplus E_8(-1)\oplus U \oplus \langle-2 \rangle$$ 
and
that of
transcendental lattice  ${\rm Tr}(S(X,0))$ is given by $$U\oplus\langle2\rangle=:A_X.$$

{\rm (2)} The projective monodromy group of the multivalued period mapping for $\mathcal{F}_X'$ is
isomorphic to $PO^+(A_X,\mathbb{Z})$.
\end{thm}

From the period mapping $\Phi$ in (\ref{EPeriod}),
the system of coordinates $(z_1,z_2)$ in (\ref{upperhalf}),
 Corollary \ref{Cor Delta} and the above theorem,
we obtain a multivalued period mapping $\Phi_X$ for $ \mathcal{F}_X$ such that
\begin{eqnarray}\label{PeriodDelta}
j^{-1} \circ \Phi_X: \{X | X \in \mathbb{P}^1(\mathbb{C})-\{0\} \} \rightarrow \varDelta,
\end{eqnarray}
where
$\Phi_X$ is given by
$\displaystyle
X\mapsto (\xi_1:\xi_2:\xi_3:\xi_4)=\Big(\int_{\Gamma_1}\omega:\int_{\Gamma_2}\omega:\int_{\Gamma_3}\omega:0 \Big)\in \mathcal{D}_+
$ 
satisfying the Riemann-Hodge relation
$
\displaystyle
\Big(\int_{\Gamma_1}\omega \Big)\Big(\int_{\Gamma_2}\omega \Big)+\Big( \int_{\Gamma_3}\omega \Big)^2=0.
$
The fundamental group $\pi_1(\Sigma,*)$ induces the projective monodromy group $M_X$ for $\Phi_X$.
According to the above theorem (3),
$M_X$ is isomorphic to $PO^+(A_X,\mathbb{Z})$.
From (\ref{upperhalf}),
we have the coordinate $z$ of $\varDelta\simeq \mathbb{H}$:
\begin{eqnarray}\label{z}
z= - \frac{\displaystyle \int_{\Gamma_3 } \omega}{\displaystyle \int_{\Gamma_2}\omega}.
\end{eqnarray}
Recalling (\ref{Qmap}),
we obtain an extended mapping 
$\Pi\circ j^{-1}\circ \Phi_X: \mathbb{P}^1(\mathbb{C})\rightarrow \overline{\varDelta/M_X}$.
We note $\Pi\circ j^{-1}\circ \Phi_X(0)$ is the $M_X$ orbit of $(\sqrt{-1}\infty,\sqrt{-1}\infty)$.
The action of $M_X$ on $\varDelta (\subset \mathbb{H}\times \mathbb{H})$
induces the action of $PSL(2,\mathbb{Z})$ on $\mathbb{H}$,
for we have  the coordinate $z$ in (\ref{z}).
Namely, there exist $\gamma_1,\gamma_2\in\displaystyle \pi_1 (\Sigma,*)$ such that
\begin{eqnarray}\label{loops}
\gamma_1(z)=z+1, \quad\quad\quad \gamma_2(z)=-\frac{1}{z}.
\end{eqnarray}
So, $\overline{\varDelta/M_X}$ is identified with the orbifold $\overline{\mathbb{H}/PSL(2,\mathbb{Z})}\simeq \mathbb{P}^1(\mathbb{C})$.

\begin{rem} \label{monodRem}
The  projective monodromy group $M_X\simeq PO^+(A_X,\mathbb{Z})$  of the 
period mapping $\Phi_X$ is
generated  by two elements:
 \begin{align}\label{monod}
\begin{pmatrix} 1 &-1&2 \\ 0&1&0 \\ 0&-1&1\end{pmatrix}, \quad\quad
\begin{pmatrix} 0&-1&0 \\ -1&0&0 \\ 0&0&-1 \end{pmatrix}.
\end{align}
These are induced by the 
 monodromy matrices in {\rm (\ref{monodromyUDE})}. 
\end{rem}

\subsection{The Gauss hypergeometric equation $\displaystyle{}_2E_1 \Big(\frac{1}{12},\frac{5}{12},1;t\Big)$}

We recall the Gauss hypergeometric equation
\begin{eqnarray}\label{Gauss j}
{}_2E _1\Big( \frac{1}{12}, \frac{5}{12},1 ;t\Big): t(1-t) \frac{d^2}{dt^2}u +(1- \frac{3}{2} t)\frac{d}{dt}u-\frac{5}{144}u=0.
\end{eqnarray}
The Riemann scheme of $\displaystyle {}_2E_1\Big(\frac{1}{12},\frac{5}{12},1;t\big)$ is given by
\begin{align*}
\begin{Bmatrix}
t=0 &t=1& t=\infty\\
\displaystyle 0& 0&{1}/{12}\\
\displaystyle 0 &{1}/{2} &{5}/{12}
\end{Bmatrix}.
\end{align*}

We can take the solutions $y_1(t)$ and $y_2(t)$ of  $\displaystyle {}_2E_1\Big( \frac{1}{12}, \frac{5}{12},1 ;t\Big)$ such that the inverse mapping of  
 the Schwarz mapping 
\begin{eqnarray}\label{z0}
\begin{matrix}
\sigma&: &  \mathbb{C}-\{0,1\}  &\rightarrow & \mathbb{H} \\
 &;&  t &\mapsto  &\displaystyle \sigma (t)=\frac{y_2(t)}{y_1(t)}=z_0 
 \end{matrix}
\end{eqnarray}
is given by
\begin{eqnarray}\label{inverse Schwarz}
z_0\mapsto \frac{1}{J(z_0)},
\end{eqnarray}
where $J(z)$ is the elliptic $J$ function with $\displaystyle J\Big(\frac{1+\sqrt{-3}}{2}\Big)=0,J(\sqrt{-1})=1$ and $J(\sqrt{-1}\infty)=\infty$.

\begin{rem}
The above $J$ function is given by
\begin{eqnarray}\label{J-function}
J(z)=\frac{1}{1728}\Big(\frac{1}{q}+744 +196884 q + \cdots \Big),
\end{eqnarray}
where $q=e^{2 \pi \sqrt{-1} z}$.
\end{rem}

Note that the Schwarz mapping $\sigma$ is a multivalued analytic mapping.
We can choose the single-valued branch of the Schwarz mapping $\sigma$  on $(0,1) \subset \mathbb{R}$ such that  $\sigma(t) \in \sqrt{-1} \mathbb{R}$ and 
\begin{eqnarray}\label{branchSchwarz}
\lim_{t\rightarrow +0} \sigma(t)=\sqrt{-1}\infty, \quad\quad \lim_{t\rightarrow 1-0} \sigma(t)=\sqrt{-1}.
\end{eqnarray}
Then, the single-valued branch of the solutions   $y_1(t)$ and $y_2(t)$ near $(0,1) (\subset \mathbb{R})$ is in the form
\begin{align}\label{branchform}
\begin{cases}
&y_1(t)=u_{11} (t),\\
&y_2(t)= \log(t)\cdot u_{21}(t)+ u_{22}(t),
\end{cases}
\end{align}
where $u_{jk}(t)$ are unit holomorphic functions around  $t=0$ 
and $\log$ stands for the principal value.

The projective monodromy group of $\displaystyle {}_2E_1\Big(\frac{1}{12},\frac{5}{12},1;t\Big)$ is isomorphic to 
$PSL(2,\mathbb{Z})$.
In other words,
the action of the fundamental group $\pi_1(\mathbb{P}^1(\mathbb{C})-\{0,1,\infty\},*)$ on $\mathbb{H}=\displaystyle\Big\{z_0=\frac{y_2}{y_1}\Big\}$ is generated by the two actions
\begin{eqnarray}\label{z0 monod}
z_0\mapsto z_0+1, \quad\quad\quad z_0\mapsto \displaystyle-\frac{1}{z_0},
\end{eqnarray}
if we normalize a basis $y_1,y_2$ of the solutions of $\displaystyle {}_2 E _1\Big(\frac{1}{12},\frac{5}{12},1;t\Big)$ around a base point.

\begin{rem}
 The projective monodromy group
 for the system $(y_2^2(t); -y_1^2(t);y_1(t) y_2(t))$  is 
 generated by the two matrices in {\rm (\ref{monod})}.
\end{rem}

\subsection{The period differential equation}

In this subsection, we determine the period differential equation for the family $\mathcal{F}_X$.
Then, considering the solutions of this period differential equation, we shall obtain the expression of $X$ using the coordinate $z$ in (\ref{z}).

\begin{prop}
On the locus $\{Y=0\}$,
the period differential equation {\rm (\ref {periodUDE})}  is restricted to the following ordinary differential equation of rank $4$:
\begin{align}\label{rest4}
\notag
 &\frac{d^4}{dX^4} u +\frac{3 (243X^2-4060X+2000)}{2X(81X^2-1155X+1000)}\frac{d^3}{dX^3 }u\\
 &\quad+\frac{2034X^2 -40680X+8000}{8X^2(81X^2-1155X +1000)}\frac{d^2}{dX^2}u+\frac{15(3X-80)}{8X^2 (81X^2-1155X +1000)}\frac{d}{dX}u=0.
\end{align}
\end{prop}

\begin{proof}
Recalling the period differential equation (\ref{periodUDE}),
set
\begin{align*}
\begin{cases}
& E_1 u =L_1 u_{XY}+ A_1 u_X+ B_1 u_Y +P_1 u,\\
& E_2 u =M_1 u_{XY}+C_1 u_X +D_1 u_Y +Q_1 u.
\end{cases}
\end{align*}
Deriving these equations, we have the system of equations
\begin{align*}
\begin{cases}
\vspace{2mm}
&\displaystyle u_{XX}= E_1 u,\quad u_{XXX}=\frac{\partial}{\partial X} E_1 u,\quad u_{XXY}=\frac{\partial }{\partial Y} E_1 u,\quad u_{XXXX}=\frac{\partial^2}{\partial X^2} E_1 u,\quad u_{XXXY}=\frac{\partial^2}{\partial X \partial Y} E_1 u,\\
&\displaystyle u_{YY}=E_2 u, \quad u_{XYY}=\frac{\partial}{\partial X} E_2 u,\quad u_{YYY}=\frac{\partial}{\partial Y} E_2u,\quad u_{XXYY}=\frac{\partial^2}{\partial Y^2} E_1 u=\frac{\partial^2}{\partial X^2} E_2 u.
\end{cases}
\end{align*} 
Our periods satisfy this system.
From this system, canceling the terms $u_Y,u_{XY},u_{YY},u_{XXY},u_{XYY},u_{YYY},$  $u_{XXXY}$ and $u_{XXYY}$, we can obtain the differential equation
$$
a_4(X,Y) u_{XXXX}+a_3 (X,Y) u_{XXX}+a_2(X,Y) u_{XX}+a_1(X,Y) u_X +a_0(X,Y) u=0,
$$
where $a_j(X,Y)$ $(j=1,2,3,4)$ is a polynomial in $X$ and $Y$.
Putting $Y=0$, we have (\ref{rest4}).
\end{proof}

Set 
$$
\check{\eta}_j(X)=\int_{\Gamma_j} \omega \quad\quad\quad (j\in{1,2,3}).
$$
The equation (\ref{rest4}) has the $4$-dimensional space of solutions
generated by $\check{\eta}_1(X),\check{\eta}_2(X),\check{\eta}_3(X)$ and $1$.
The Riemann scheme of $(\ref{rest4})$ is geven by
\begin{eqnarray*}
\begin{Bmatrix}
X=0 & X=25/27 & X=40/3 & X=\infty\\
0&0&0&0\\
1&1/2&1&-5/6\\
1&1&2&-1/2\\
1&2&4&-1/6
\end{Bmatrix}.
\end{eqnarray*}

Setting $\displaystyle X=\frac{25}{27} t $,
the equation (\ref{rest4}) is transposed to
 $$
 W_4 u=0,
 $$
 where
 $$
 W_4=\frac{d^4}{dt^4} +\frac{1620t^3-29232t^2+15552t}{72t^2(t-1)(5t-72)}\frac{d^3}{dt^3}+\frac{565 t^2-12204 t+2592}{36t^2(t-1)(5t-72)}\frac{d^2}{dt^2}+\frac{25t-720}{72t^2(t-1)(5t-72)}\frac{d}{dt}.
$$

Straightforward calculation shows the following.

\begin{prop} Set
\begin{align*}
W_3=\frac{d^3}{dt^3}+\frac{3}{2(t-1)}\frac{d^2}{dt^2}+\frac{5t -36}{36t^2 (t-1)}\frac{d}{dt}+\frac{72-5t}{72 t^3 (t-1)},
\quad
W_1=\frac{d}{dt} +\frac{15t^2 -298t +216}{t(t-1)(5t -72)}.
\end{align*}
It holds that
\begin{eqnarray}\label{W4=W1W3}
W_4= W_1 \circ W_3.
\end{eqnarray}
\end{prop}

Set $\displaystyle \eta_j(t)=\check{\eta}_j\Big(\frac{25}{27}t \Big)$ for $j\in\{1,2,3\}$.

\begin{prop}
The periods $\eta_1(t),\eta_2(t)$ and $\eta_3(t)$ are the solutions of 
$$
W_3 u=0
$$
satisfying 
\begin{eqnarray}\label{P2ji}
\eta_1 \eta_2 +\eta_3^2=0.
\end{eqnarray}
\end{prop}

\begin{proof}
Let $V=\langle \eta_1,\eta_2,\eta_3 \rangle_\mathbb{C}$ 
and $V'=\langle W_3\eta_1,W_3\eta_2,W_3\eta_3 \rangle_\mathbb{C}$.
Since the linear mapping $W_3:V\rightarrow V'$ given by 
$f \mapsto W_3 f$ is monodromy-equivalent and $V$ is an irreducible representation,
according to Schur's lemma,
we have $V\simeq V'$ or $V'=\{0\}$. 
It follows
from (\ref{W4=W1W3}) that  $V'\subset {\rm Ker}(W_1)$.
Because ${\rm dim}({\rm Ker}(W_1))=1$, we have $V'=\{0\}$.

For $t\mapsto (\eta_1(t):\eta_2(t):\eta_3(t))$ is the period mapping $\Phi_X$, the relation (\ref{P2ji}) is clear.
\end{proof}

\begin{prop}\label{direct}
If $u_1$ and $u_2$ are  solutions of $\displaystyle {}_2E_1\Big(\frac{1}{12},\frac{5}{12},1;t\Big)$, then $t u_1^2(t),  t u^2_2(t)$ and $t u_1(t) u_2(t)$ are  solutions
of the period differential equation $W_3 u =0$. 
\end{prop}

\begin{proof}
Take any solutions of $\displaystyle {}_2E_1\Big(\frac{1}{12},\frac{5}{12},1;t\Big)$ $u_1(t)$ and $u_2(t)$.
For $j\in\{1,2 \}$,
\begin{eqnarray}\label{Gauss1}
u_j''=\frac{\displaystyle 1-{3t}/{2}}{t(t-1)} u_j' -\frac{5}{144t(t-1)} u_j,
\end{eqnarray}
then
\begin{align}\label{Gauss2}
u_j^{(3)}=\frac{535 t^2 -715 t +288}{144 t^2 (t-1)^2} u_j' +\frac{5(7t-4)}{288 t^2(t-1 )^2}u_j.
\end{align}
Here, by  a straightforward calculation, we have
\begin{align}\label{sym.exp}\notag
W_3 (t u_1 u_2)&=\frac{5}{72 t (t-1)} u_1 u_2+\frac{113 t -36 }{36 t (t-1)} (u_1'u_2+u_1u_2')+\frac{3(3t-2)}{t-1} u_1' u_2' \\
&\quad\quad+\frac{3(3t-2)}{2(t-1)} (u_1''u_2+u_1u_2'')+3t(u_1'u_2''+u_1''u_2')+t(u_1^{(3)}u_2+u_1u_2^{(3)}).
\end{align}
Substituting (\ref{Gauss1}) and (\ref{Gauss2}) for (\ref{sym.exp}), we have $W_3(t u_1 u_2)=0$.
\end{proof}

\begin{rem}
According to {\rm (\ref{rest4})}, the derivation $\displaystyle \frac{d}{dt} \eta_j$ $(j=1,2,3)$ of  the period is a solution of the equation
\begin{align}
\label{restdiff3}
\frac{d^3}{dt^3}v +\frac{1620t^3-29232t^2+15552t}{72t^2(t-1)(5t-72)}\frac{d^2}{dt^2} v
 +\frac{1130t^2-24408 t+5184}{72t^2(t-1)(5t-72)}\frac{d}{dt}v+\frac{25t-720}{72t^2(t-1)(5t-72)}v=0.
 \end{align}
Then, set
\begin{eqnarray*}
S(t)={}_3F_2\Big(\frac{1}{6},\frac{1}{2},\frac{5}{6};1,1;t\Big)+\frac{1}{5} {}_3F_2\Big(\frac{7}{6},\frac{1}{2},\frac{5}{6};1,1;t\Big),
\end{eqnarray*}
 where
 ${ }_3F_2$ is the generalized hypergeometric series:
 $$
{}_3F_2(a_1,a_2,a_3;b_1 ,b_2;t)=\sum_{t=0}^\infty \frac{(a_1,n)(a_2,n)(a_3,n)}{(b_1,n)(b_2,n)n!}t^n.
$$
We see that $S(t)$ is a holomorphic solution of {\rm (\ref{restdiff3})} around $t=0$.
The indefinite integral of $S(t)$ with the integral constant $0$ is given by
\begin{align}\notag
&\displaystyle t\cdot {}_3F_2\Big(\frac{1}{6},\frac{1}{2},\frac{5}{6};1,2;t\Big) +\frac{1}{5} t\cdot {}_3F_2\Big(\frac{7}{6},\frac{1}{2},\frac{5}{6};1,2;t\Big)\\
\notag
&=\displaystyle \frac{6}{5}t\cdot {}_3F_2 \Big(\frac{1}{6},\frac{1}{2},\frac{5}{6};1,1;t\Big)
=\frac{6}{5}t \cdot\Big({}_2F_1\Big(\frac{1}{12}, \frac{5}{12},1;t\Big)\Big)^2.
\end{align}
Here, we applied Clausen's formula.
From the above proposition, this gives a holomorphic solution of $W_3 u=0$ around $t=0$.
\end{rem}

Let $y_1(t)$ and $y_2(t)$ are the single-valued branches of the solutions of $\displaystyle {}_2E_1\Big(\frac{1}{12},\frac{5}{12},1;t\Big)$ near $(0,1)\subset \mathbb{R}$  given in (\ref{branchSchwarz}).
Let
$$
s_1(t)=t y_1^2(t), \quad s_2(t)=t y_1(t)y_2(t),\quad s_3(t)=ty_2^2(t).
$$
Note that, if $t\in(0,1)\subset \mathbb{R}$, we have
\begin{align}\label{log-term}
\begin{cases}
&s_1(t)=t\cdot v_{11}(t),\\
&s_2(t)=t\cdot(\log(t)v_{21}(t)+ v_{22}(t)),\\
&s_3(t)=t\cdot (\log^2(t)v_{31} (t) + \log(t) v_{32}(t) +v_{33}(t)),
\end{cases}
\end{align}
where $v_{jk}(t)$ are unit holomorphic functions around $t=0$.
Moreover, they satisfy
\begin{eqnarray}\label{2ji}
-s_1(t)  s_3(t) +s_2^2(t)=0.
\end{eqnarray}

\begin{lem}
A   branch of the multivalued analytic mapping $t\mapsto (\eta_1(t):\eta_2(t):\eta_3(t))$ satisfies
$$
\big(\eta_1(t):\eta_2(t):\eta_3(t)\big)=\big(s_3(t):-s_1(t):s_2(t)\big)\in\mathbb{P}^2(\mathbb{C}).
$$
\end{lem}

\begin{proof}
Because we have Proposition \ref{noteH} (1)
and the coordinate $z$ in (\ref{z}), we take the single-valued branch of the multivalued period mapping 
$t\mapsto \big(\eta_1(t):\eta_2(t):\eta_3(t)\big)$ on $t\in (0,1)\subset\mathbb{R}$ such that
\begin{eqnarray}\label{Pcusp}
\lim_{t\rightarrow +0} -\frac{\eta_3(t)}{\eta_2(t)}=\sqrt{-1}\infty.
\end{eqnarray}
In this proof, we consider  $\eta_1(t),\eta_2(t)$ and $\eta_3(t)$ near $(0,1)(\subset\mathbb{R})$.

According to Proposition \ref{direct}, we have
$$
\eta_j(t)=\sum_{k=1}^3 a_{jk} s_k(t)         \hspace{2.5cm} (j=1,2,3),
$$
where $a_{jk}$ $(j,k=1,2,3)$ are constants.
Since we have (\ref{Pcusp}), we obtain $a_{23}=0$.
So, it follows that $\eta_2(t)=a_{21} s_1(t)+a_{22} s_2(t)$.
From (\ref{log-term}), we see that $\eta_1(t)\eta_2(t)$ does not contain $\log^4(t)$. Then, from (\ref{P2ji}), we have $a_{33}=0$.
Recalling (\ref{Pcusp}) again, we obtain $a_{22}=0$.
Because we consider $y\mapsto \big(\eta_1(t):\eta_2(t):\eta_3(t)\big)\in\mathbb{P}^2(\mathbb{C})$, we assume that $a_{21}=-1$.
Then, the single-valued branches $\eta_j(t)$ $(j=1,2,3)$ are in the form
\begin{align*}
\begin{cases}
&\eta_{1}(t) = a_{11} s_1(t)+a_{12} s_2(t)+a_{13} s_{3}(t),\\
&\eta_{2}(t)=-s_{1}(t),\\
&\eta_{3}(t)=a_{31} s_1(t)+a_{32} s_2(t).
\end{cases}
\end{align*}
Hence, using (\ref{z0}), the coordinate $z$ in (\ref{z}) is given by
\begin{eqnarray*}
z=a_{32}\frac{s_2(z)}{s_1(z)}+ a_{31} =a_{32} z_0 +a_{31}.
\end{eqnarray*}

Considering the actions of $\pi_1(\mathbb{P}^1(\mathbb{C})-\{0,1,\infty\})$ on $z=-\displaystyle \frac{\eta_3}{\eta_2}$ -space  in (\ref{loops}) and $z_0=\displaystyle \frac{y_2}{y_1}$ -space in (\ref{z0 monod}),
we have $a_{31}=0$ and $a_{32}=1$.

Therefore, using (\ref{P2ji}) again, 
we obtain
\begin{eqnarray*}
\eta_1(t)= s_3(t),\quad\quad \eta_2(t)=-s_1(t),\quad\quad \eta_3(t)=s_2(t).
\end{eqnarray*}
\end{proof}

\begin{cor}\label{SchwarzLemma}
A coordinate $z$ in {\rm (\ref{z})} of the diagonal $\varDelta $ $(\simeq\mathbb{H})$ is equal to
\begin{eqnarray*}
z=\frac{y_2(t)}{y_1(t)}. 
\end{eqnarray*}
\end{cor}

\begin{proof}
From  the above lemma, this is clear.
\end{proof}

\begin{thm} \label{j-functionThm}
The inverse of the multivalued period mapping $j^{-1}\circ \Phi_X:X \mapsto (z,z)$ in {\rm (\ref{PeriodDelta})} is given by
\begin{eqnarray*}
X(z,z)=\frac{25}{27} \cdot \frac{1}{J(z)}.
\end{eqnarray*}
\end{thm}

\begin{proof}
 From the above Corollary and the inverse Schwarz mapping (\ref{inverse Schwarz}), 
we have
$
t(z)=\displaystyle \frac{1}{J(z)}.$
Therefore, we obtain
$$
X(z,z)=\frac{25}{27} \cdot t(z)=\frac{25}{27}\cdot\frac{1}{J(z)}.
$$
\end{proof}

\section{The theta expressions of $X$ and $Y$}

First, we recall the classical elliptic functions. Let $z\in \mathbb{H}$.

The classical Eisenstein series are given by
$$
G_2(z)=60\sum_{(0,0)\not=(m,n)\in\mathbb{Z}^2} \frac{1}{(m z +n)^4}, \quad\quad G_3(z)=140\sum_{(0,0)\not=(m,n)\in\mathbb{Z}^2} \frac{1}{(mz+n)^6} .
$$
$G_2(z)$ ($G_3(z)$, resp.) is a modular form  of weight $4$ ($6$, resp.) for $PSL(2,\mathbb{Z})$.
The ring of modular forms for $PSL(2,\mathbb{Z})$ is $\mathbb{C}[G_2,G_3]$. 
We have 
$\displaystyle
G_2(\sqrt{-1}\infty)=\frac{4 \pi^4}{3}$ and $\displaystyle G_3(\sqrt{-1}\infty) =\frac{8 \pi^6}{27}.
$
Let $\displaystyle E_4(z)=\frac{3}{4\pi^4} G_2(z)$ and $\displaystyle E_6(z)=\frac{27}{8 \pi^6} G_3(z)$ be the normalized Eisenstein series.
The discriminant form is 
$$
\Delta (z)=G_2^3(z)-27G_3^2(z).
$$
We have $\Delta(\sqrt{-1}\infty)=0$.
This is a cusp form of weight $12$.  
The cusp form of weight $12$ is  $\Delta$ up to a constant factor. 
The $J$ function in (\ref{J-function}) is given by
\begin{eqnarray}\label{J-Def}
J(z)=\frac{G_2^3(z)}{G_2^3(z)-27 G_3^2(z)}=\frac{G_2^3(z)}{\Delta(z)}.
\end{eqnarray}
The field of modular functions for the  modular group $PSL(2,\mathbb{Z})$ is  $\mathbb{C}(J(z)).$

For $a,b \in \{ 0,1 \} $, the Jacobi theta constants are defined by
$$
\vartheta_{ab}(z)=\sum_{n\in\mathbb{Z}} {\rm exp}\Big( \sqrt{-1}\pi \Big( n+\frac{a}{2} \Big)^2 z + 2\sqrt{-1}\pi \Big( n+\frac{a}{2}\Big)\frac{b}{2} \Big)
$$
for $(a,b)=(0,0),(0,1)$ and $(1,0)$.
The functions  $\vartheta_{00}^4(z),\vartheta_{01}^4(z)$ and $\vartheta_{10}^4(z)$ are the modular forms of weight $2$ for the  principal congruence subgroup $\displaystyle \Gamma(2)=\Big\{ \begin{pmatrix} \alpha &\beta\\ \gamma&\delta\end{pmatrix}| \alpha\equiv \delta \equiv 1, \beta\equiv\gamma\equiv 0 \quad({\rm mod}\hspace{0.2cm}2) \Big\}.$
The ring of modular forms for $\Gamma(2)$ is
$$
\mathbb{C}[\vartheta_{00}^4, \vartheta_{01}^4, \vartheta_{10}^4 ]/(\vartheta_{01}^4+\vartheta_{10}^4=\vartheta_{00}^4)=\mathbb{C}[\vartheta_{00}^4,\vartheta_{01}^4].
$$

We note that
$$
\frac{1}{1728} \Big(\frac{3}{4\pi^4}\Big)^3 \Delta (z)= \frac{1}{2^8} \vartheta_{00}^8(z) \vartheta_{01}^8(z) \vartheta_{10}^8 (z).
$$

Next, we survey the theta constants for Hilbert modular forms for $\mathbb{Q}(\sqrt{5})$.
They are introduced by M\"{u}ller \cite{Muller}.

Set 
$$
\mathfrak{S}_2=\{Z\in {\rm Mat} (2,2) |  {}^tZ =Z ,  {\rm Im}(Z)>0 \}.
$$
This is the Siegel upper half plane consisting of $2\times 2$ complex matrices.
For $a,b \in \{0,1\}^2$ with ${}^t a b\equiv0 \hspace{1mm}({\rm mod}2)$,
set
\begin{align*}
\vartheta (Z;a,b)=\sum _{g\in \mathbb{Z}^2} {\rm exp}\Big(\pi \sqrt{-1} \big( {}^t\big( g+\frac{1}{2} a\big) Z \big( g+\frac{1}{2}a \big)+{}^tgb\big)\Big).
\end{align*}

We use the mapping $\psi:\mathbb{H}\times \mathbb{H} \rightarrow \mathfrak{S}_2$
given by
\begin{align*}
 (z_1,z_2) = \zeta & \mapsto \begin{pmatrix}\displaystyle {\rm Tr}\Big(\frac{\varepsilon \zeta}{\sqrt{5}}\Big) & \displaystyle {\rm Tr}\Big(\frac{ \zeta} {\sqrt{5}}\Big) \\  \displaystyle {\rm Tr}\Big(\frac{ \zeta} {\sqrt{5}}\Big) & \displaystyle {\rm Tr}\Big(-\frac{\varepsilon ' \zeta}{\sqrt{5}}\Big) \end{pmatrix}\\
&
\quad\quad=\frac{1}{2 \sqrt{5}}\begin{pmatrix}
(1+\sqrt{5})z_1 -(1-\sqrt{5})z_2 & 2(z_1-z_2)\\
2(z_1-z_2) & (-1+\sqrt{5})z_1 +(1+\sqrt{5})z_2
\end{pmatrix},
\end{align*}
where $\displaystyle \varepsilon =\frac{1+\sqrt{5}}{2}$.

\begin{rem}
Set
\begin{align*}
\mathcal{N}_5=\Big\{\begin{pmatrix} \sigma_1 & \sigma_2 \\ \sigma_2 & \sigma_3 \end{pmatrix} \in \mathfrak{S}_2 \Big| -\sigma_1 +\sigma_2 +\sigma_3=0 \Big\}.
\end{align*}
Let $p$ be the  canonical projection $\mathfrak{S}_2 \rightarrow \mathfrak{S}_2 /Sp(4,\mathbb{Z})$.
Then, the Humbert surface $\mathcal{H}_5=p(\mathcal{N}_5)$ of invariant $5$ gives the moduli space of  principally polarized Abelian surfaces $A$ such that $\mathbb{Q}(\sqrt{5})\subset {\rm End}(A)\otimes \mathbb{Q}$.
We note that the above  $\psi$
is a mapping $\mathbb{H}\times\mathbb{H}\rightarrow \mathcal{N}_5$.
\end{rem}

For $j\in\{0,1,\cdots,9\}$, we set
\begin{align*}
\theta_j (z_1,z_2)=\vartheta (\psi(z_1,z_2 ) ;a,b),
\end{align*}
where the correspondence between $j$ and $(a,b)$ is given by Table 1.
\begin{table}
\center
{\small
\begin{tabular}{lcccccccccc}
\toprule
$j$&$0$&$ 1$ &$2$ &$3$&$4$&$5$&$6$&$7$&$8$&$9$  \\
\midrule
${}^t a$& $(0,0)$ &$(1,1)$ &$(0,0)$&$(1,1)$&$(0,1)$&$(1,0)$&$(0,0)$&$(1,0)$&$(0,0)$ &$(0,1)$ \\
${}^t b$ &$(0,0)$&$(0,0)$&$(1,1)$&$(1,1)$& $(0,0) $& $(0,0)$& $(0,1)$& $(0,1)$ & $(1,0)$&$(1,0)$\\
\bottomrule
\end{tabular}
}
\caption{The correspondence between $j$ and $(a,b)$.}
\end{table}
These theta constants are  holomorphic functions on $\mathbb{H}\times \mathbb{H}$.

Let $a\in \mathbb{Z}$ and $j_1,\cdots,j_r \in\{0,\cdots,9\}$. 
We set  $\theta_{j_1,\cdots,j_r}^a  = \theta_{j_1}^a \cdots \theta_{j_r} ^a$.

Set
$
s_5=2^{-6} \theta_{0123456789}.
$
This is an alternating modular form of weight $5$. 
The following $g_2$ ($s_6,s_{10},s_{15}$, resp.) is   a symmetric Hilbert modular form of weight $2$ ($6,10,15,$ resp.) for  $\mathbb{Q}(\sqrt{5})$:
\begin{align}\label{MTheta}
\begin{cases}
&g_2= \theta_{0145}-\theta_{1279}-\theta_{3478}+\theta_{0268} +\theta_{3569},\\
&s_6=2^{-8} (\theta_{012478}^2 +\theta_{012569}^2 +\theta_{034568}^2 + \theta_{236789}^2 +\theta_{134579}^2),\\
&s_{10}=s_5^2 = 2^{-12} \theta_{0123456789}^2,\\
&s_{15}=-2^{-18} (\theta_{07}^9 \theta_{18}^5 \theta_{24} - \theta_{25}^9 \theta_{16}^5 \theta_{09} +\theta_{58}^9 \theta_{03}^5 \theta_{46} -\theta_{09}^9\theta_{25}^5\theta_{16} +\theta_{09}^9\theta_{16}^5\theta_{25} -\theta_{67}^9 \theta_{23}^5 \theta_{89} \\
&\quad\quad\quad\quad \quad
+\theta_{18}^9\theta_{24}^5 \theta_{07} -\theta_{24}^9 \theta_{18}^5 \theta_{07}
-\theta_{46}^9 \theta_{03}^5 \theta_{58} - \theta_{24}^9 \theta_{07}^5 \theta_{18}
-\theta_{89}^9\theta_{67}^5\theta_{23} -\theta_{07}^9\theta_{24}^5\theta_{18}\\
&\quad\quad\quad\quad \quad
+\theta_{89}^9\theta_{23}^5 \theta_{67} -\theta_{49}^9\theta_{13}^5 \theta_{57}
+\theta_{16}^9\theta_{09}^5\theta_{25} -\theta_{03 } ^9 \theta_{46}^5 \theta_{58}
+\theta_{16}^9\theta_{25}^5 \theta_{09} -\theta_{46}^9\theta_{58}^5\theta_{03}\\
&\quad\quad\quad\quad \quad
-\theta_{25}^9 \theta_{09}^5 \theta_{16} -\theta_{57}^9 \theta_{49}^5 \theta_{13}
+\theta_{67}^9\theta_{89}^5\theta_{23} +\theta_{58} ^9 \theta_{46}^5 \theta_{03}
+\theta_{57}^9  \theta_{13}^5 \theta_{49} -\theta_{23}^9 \theta_{89}^5 \theta_{67}\\
&\quad\quad\quad\quad \quad
+\theta_{18}^9 \theta_{07}^5 \theta_{24} +\theta_{03}^9 \theta_{58}^5 \theta_{46}
+\theta_{23}^9 \theta_{67}^5 \theta _{89} +\theta_{49}^9 \theta_{57}^5 \theta_{13}
-\theta_{13}^9 \theta_{57}^5 \theta_{49 } +\theta_{13}^9 \theta_{49}^5 \theta_{57}).
   \end{cases}
\end{align}

\begin{prop} {\rm(\cite{Muller} Satz 1)} \label{MullerTheta}
{\rm (1) }The ring of the symmetric Hilbert modular forms for $\mathbb{Q}(\sqrt{5})$ is given by
$$
\mathbb{C}[g_2,s_6,s_{10},s_{15}]/(M(g_2,s_6,s_{10},s_{15})=0),
$$
where
\begin{align}\label{Mullerrelation}
\notag&M(g_2,s_6,s_{10},s_{15})\\
&=
s_{15}^2-\Big( 5^5 s_{10}^3 -\frac{ 5^3}{2} g_2 ^2s_6 s_{10}^2 +\frac{1}{2^{4}} g_2^5 s_{10}^2 +\frac{3^2 \cdot5^2}{2} g_2 s_6^3 s_{10 }- \frac{1}{2^{3}} g_2^4 s_6^2 s_{10} -2 \cdot 3^3 s_6^5  +\frac{1}{ 2^{4}} g_2^3s_6^4\Big).  
\end{align}

{\rm (2)} The ring of the  Hilbert modular forms for $\mathbb{Q}(\sqrt{5})$ is given by
$$
\mathbb{C}[g_2,s_5,s_6,s_{15}]/(M(g_2,s_5^2, s_6,s_{15})=0).
$$
\end{prop}


\begin{prop}\label{boundary} {\rm (\cite{Muller} pp.244-245)}
M\"uller's modular forms satisfy
\begin{align*}
\begin{cases}
g_2(i\infty,i\infty)=1,\\
\displaystyle s_6(z,z)=\frac{2}{1728}\Big(\frac{3}{4 \pi^4 }\Big)^3\Delta(z)=\frac{1}{2^7} \vartheta_{00}^8 (z) \vartheta_{01}^8(z) \vartheta_{10}^8 (z),\\
s_{10}(z,z)=0.
\end{cases}
\end{align*}
Especially,  the relations
\begin{align*}
\begin{cases}
&\vspace{2mm} \displaystyle \frac{4 \pi^4}{3} g_2(z,z)= \frac{4 \pi^4}{3}E_4(z)=  G_2(z), \\
&2^{11} \pi^{12} s_6(z,z) \displaystyle
=G_2^3(z)-27 G_3^2(z)=\Delta(z)
\end{cases}
\end{align*}
hold.
\end{prop}

\vspace{5mm}

Now, we obtain  the  theta expressions   of the parameters $X$ and $Y$ for the family $\mathcal{F}$.
According  to Proposition \ref{KleinH},
$\{X=\displaystyle \frac{\mathfrak{B}}{\mathfrak{A}^3} ,Y= \displaystyle \frac{\mathfrak{C}}{\mathfrak{A}^5}\}$ gives
a system of generators of
symmetric  Hilbert modular functions for $\mathbb{Q}(\sqrt{5}).$
From Theorem \ref{thmC},
 the inverse correspondence
 $(z_1,z_2)\mapsto (X(z_1,z_2),Y(z_1,z_2))$
  of  the multivalued period mapping  for $\mathcal{F}$
defines  the pair of  Hilbert modular  functions of variables  $z_1$ and $z_2$
in (\ref{upperhalf}).
In the following argument,  
we shall obtain the expression of  $X(z_1,z_2)$ and $Y(z_1,z_2)$ as the quotients of  M\"uller's modular forms.

 For our argument, we set $\displaystyle Z=\frac{\mathfrak{D}^2}{\mathfrak{A}^{15}}$.
 This defines a symmertic Hilbert modular function for $\mathbb{Q}(\sqrt{5})$ also.
  
\begin{lem}
The modular functions $X(z_1,z_2),Y(z_1,z_2)$ and $Z(z_1,z_1)$ have the expressions
\begin{align}\label{k1k2k3}
\begin{cases}
\vspace{2mm}
&X(z_1,z_2)=\displaystyle k_1 \frac{s_6(z_1,z_2)}{g_2^3(z_1,z_2)}, \\
\vspace{2mm}
&Y(z_1,z_2)=\displaystyle k_2 \frac{s_{10}(z_1,z_2)}{g_2^5(z_1,z_2)},\\
&Z(z_1,z_2)=\displaystyle k_3 \frac{s_{15}^2(z_1,z_2)}{g_2^{15}(z_1,z_2)},
\end{cases}
\end{align}
for some $k_1,k_2$ and $k_3\in\mathbb{C}$.
\end{lem}

\begin{proof}
Since $\displaystyle X=\frac{\mathfrak{B}}{\mathfrak{A}^3}$,
the modular function $X$ is given by the quotient of  Hilbert modular forms of weight $6$ and its denominator is the cube of a Hilbert modular form of weight $2$.
Note that,  a Hilbert modular form of weight $2$  is equal to $g_2$ up to a constant factor.
Then, we have 
$$
X(z_1,z_2)=\frac{k_{11} s_6(z_1,z_2)+ k_{12} g_2^3(z_1,z_2)}{k_{13} g_2^3(z_1,z_2) },
$$
where $k_{11},k_{12}$ and $k_{13}$ are constants.
Recalling Proposition \ref{noteH}  (1), we have $X(\sqrt{-1}\infty, \sqrt{-1}\infty)=0$.  
Then, from Proposition \ref{boundary}, we obtain $k_{12}=0$ and
$$
X(z_1,z_2)=k_1 \frac{s_6(z_1,z_2)}{g_2^3(z_1,z_2)}.
$$

Since $\displaystyle Y=\frac{\mathfrak{C}}{\mathfrak{A}^5}$, 
the modular function $Y$ is given by the quotient of  Hilbert modular forms of weight $10$.  Its denominator is the $5$-th power of a modular form of weight $2$.
Then,  we have
$$
Y(z_1,z_2)=\frac{k_{21} s_{10}(z_1,z_2)+ k_{22} g_2^5(z_1,z_2)+ k_{23} g_2^2(z_1,z_2)s_6(z_1,z_2)}{k_{24}g_2^5(z_1,z_2)},
$$
where $k_{21},k_{22},k_{23}$ and $k_{24}$ are constants.
By Proposition \ref{noteH} (3), we have $Y(z,z)=0$.
According to (\ref{MTheta}) and Proposition \ref{boundary}, if a modular form $g$  of weight $10$ vanishes on the diagonal $\varDelta$,  then we have  $g={\rm const} \cdot s_{10} $.
So, it holds that $k_{22}=k_{23}=0$.
Therefore,  we obtain
$$
Y(z_1,z_2)=k_2 \frac{s_{10}(z_1,z_2)}{g_2^5(z_1,z_2)}.
$$

Recalling  Proposition \ref{KleinH} (2), we note that $\mathfrak{D}$ defines a  symmetric Hilbert modular form of weight $15$.
Since $\displaystyle Z=\frac{\mathfrak{D}^2}{\mathfrak{A}^{15}}$,
the modular function $Z$ is given by the quotient of  modular forms of weight $30$.  Its denominator is the $15$-th power of a modular form of weight $2$ and its numerator is given by the square of a symmetric modular form of weight $15$.
According to Proposition \ref{MullerTheta} (2), a symmetric modular form of weight $15$ is given by ${\rm const}\cdot s_{15}$.
Then, we have
$$
Z(z_1,z_2)=k_3\frac{s_{15}^2(z_1,z_2)}{g_2^{15}(z_1,z_2)}.
$$
 \end{proof}

\begin{thm}\label{main}
The inverse correspondence of the multivalued period mapping $j^{-1}\circ \Phi :(X,Y)\mapsto (z_1,z_2)$ in {\rm (\ref{upperhalf}) } for the family $\mathcal{F}$ is given by the quotient of M\"uller's modular forms:
\begin{align*}
\begin{cases}
&\displaystyle X(z_1,z_2)=2^5 \cdot 5^2 \cdot \frac{s_6(z_1,z_2)}{g_2^3(z_1,z_2)},\\
&\displaystyle Y(z_1,z_2)=2^{10} \cdot 5^5  \cdot\frac{s_{10}(z_1,z_2)}{g_2^5 (z_1,z_2)}.
\end{cases}
\end{align*}
\end{thm}

\begin{proof}
First, we obtain the expression of $X$. To obtain it, we determine  the constant $k_1$ in (\ref{k1k2k3}).
Due to Theorem \ref{j-functionThm}, (\ref{J-Def}) and Proposition \ref{boundary}, we have   
$$
X(z,z)=\frac{25}{27}\cdot\frac{1}{J(z)}=\frac{25}{27}\cdot \frac{\displaystyle 2^{11} \pi^{12} s_6(z,z) } {\displaystyle \Big(\frac{4\pi^4}{3}\Big)^3 g_2^3(z,z) } = 2^5\cdot  5^2\cdot\frac{s_6(z,z)}{g_2^3(z,z)}.
$$
So, we obtain $k_1=2^5\cdot 5^2$.

Next, we determine the constant $k_3$ in (\ref{k1k2k3}). 
By (\ref{Klein}), we have
\begin{align}\label{KleinQ0}
\notag&144 Z(z_1,z_2)=-1728 X^5(z_1,z_2)+720 X^3(z_1,z_2) Y(z_1,z_2) \\
&\hspace{3cm} -80 X(z_1,z_2) Y^2(z_1,z_2) 
+64 (5 X^2(z_1,z_2) -Y(z_1,z_2))^2 +Y^3(z_1,z_2).
\end{align}
Recalling that $Y(z,z)=0$, we have
\begin{align}\label{KleinQ}
\notag 144 Z(z,z)&=-1728X^5(z,z)+64 \cdot 25 \cdot X^4(z,z)\\
&=- 2^{26} \cdot 5^{10} \cdot \Big(2^5 \cdot 3^3\cdot \frac{s_6(z,z)}{g_2^3(z,z)} -1\Big)\Big(\frac{s_6(z,z)}{g_2^3(z,z)} \Big)^4.
\end{align}
On the other hand, from (\ref{Mullerrelation}), we have
\begin{eqnarray}\label{MullerQ}\notag
&\displaystyle \frac{s_{15}^2(z_1,z_2) }{g_2^{15}(z_1,z_2)}\displaystyle=5^5 \Big(\frac{s_{10}(z_1,z_2)}{g_2^5(z_1,z_2)}\Big)^3 - \frac{5^3}{2 } \Big(\frac{s_6(z_1,z_2)}{g_2^3(z_1,z_2)}\Big)\Big( \frac{s_{10}(z_1,z_2)}{g_2^5(z_1,z_2)}\Big)^2 +\frac{3^2\cdot 5^2}{2} \Big(\frac{s_6(z_1,z_2)}{g_2^3(z_1,z_2)}\Big)^2\Big( \frac{s_{10}(z_1,z_2)}{g_2^5(z_1,z_2)}\Big) \\ 
&\displaystyle +\frac{1}{2^4}\Big(\frac{s_{10}(z_1,z_2)}{g_2^5(z_1,z_2)}\Big)^2
-\frac{1}{2^3} \Big(\frac{s_6(z_1,z_2)}{g_2^3(z_1,z_2)}\Big)^2\Big(\frac{s_{10}(z_1,z_2)}{g_2^5(z_1,z_2)} \Big)-2\cdot 3^3 \Big(\frac{s_6(z_1,z_2)}{g_2^3(z_1,z_2)}\Big)^5+\frac{1}{2^4}\Big( \frac{s_6(z_1,z_2)}{g_2^3(z_1,z_2)} \Big)^4.
\end{eqnarray}
So, because $s_{10}(z,z)=0$, we have
\begin{eqnarray}\label{diagQ}
\Big(\frac{s_{15}^2 (z,z)}{g_2^{15}(z,z)}\Big)=\frac{1}{2^4}\Big( -2^5\cdot3^3\frac{s_6(z,z)}{g_2^3(z,z)} +1\Big)\Big(\frac{s_6(z,z)}{g_2^3(z,z)}\Big)^4.
\end{eqnarray}
Since 
$$
Z(z,z) =k_3 \frac{s_{15}^2(z,z)}{g_2^{15}(z,z)},
$$
comparing (\ref{KleinQ}), (\ref{diagQ}), 
we have $k_3=2^{26}\cdot 5^{10} \cdot 3^{-2}$. 

Finally, from (\ref{KleinQ0}), (\ref{MullerQ}), $k_1=2^{5}\cdot 5^2$ and $k_3=2^{26}\cdot 5^{10} \cdot 3^{-2}$, we have
$$
k_2=2^{10}\cdot 5^5.
$$
\end{proof}

Thus, we obtain the explicit theta expression of
the 
inverse correspondence 
$(z_1,z_2)\mapsto \big( X(z_1,z_2),Y(z_1,z_2)\big)$
of the period mapping for our family $\mathcal{F}$ of $K3$ surfaces.

\section*{Acknowledgment}
The author would like to thank Professor Hironori Shiga for helpful advises and  valuable suggestions,
and also to Professor Kimio Ueno and the members of his laboratory for kind encouragements.
He is grateful to the referee for the valuable comments to improve the manuscript.   
This work is supported by Grant-in-Aid for JSPS fellows.

\vspace{5mm}

\begin{center}
\hspace{7.7cm}\textit{Atsuhira  Nagano}\\
\hspace{7cm}\textit{ c.o. Proffessor Kimio Ueno, Department of Mathematics}\\
\hspace{7.7cm}\textit{ Waseda University}\\
\hspace{7.7cm}\textit{Okubo 3-4-1, Shinjuku-ku, Tokyo, 169-8555}\\
\hspace{7.7cm}\textit{Japan}\\
 \hspace{7.7cm}\textit{(E-mail: atsuhira.nagano@gmail.com)}
  \end{center}
 \end{document}